\documentclass[reqno,12pt]{amsart} 
\usepackage{amsmath,amssymb,amsfonts,eucal}
\usepackage{enumerate, xypic, psfrag}
\usepackage{graphicx}
\usepackage{hyperref}

\newcommand \comment[1]{}				

\newtheorem{thm}{Theorem}[section]
\newtheorem{cor}[thm]{Corollary}
\newtheorem{prop}[thm]{Proposition}
\newtheorem{lem}[thm]{Lemma}

\theoremstyle{remark}
\newtheorem{prob}[thm]{Problem}

\numberwithin{equation}{section}
\numberwithin{figure}{section}

\renewcommand \phi{\varphi}
\renewcommand \epsilon{\varepsilon}
\newcommand \eset{\emptyset}
\renewcommand \eset{\varnothing}

\newcommand \lra{\leftrightarrow}
\newcommand \inv{^{-1}}

\newcommand \bcl{\operatorname{bcl}}
\newcommand \rk{\operatorname{rk}}
\newcommand \codim{\operatorname{codim}}
\newcommand\embeds{\hookrightarrow}

\newcommand \bgr[1]{\langle#1\rangle}
\newcommand \full{^{{}^{{}_{{}_\bullet}}\!}}

\newcommand \opp{^\mathrm{op}}

\newcommand \bz{\mathbf{z}}	

\newcommand \cB{\mathcal{B}}
\newcommand \cC{\mathcal{C}}
	
\renewcommand \cL{\mathcal{L}}	
\newcommand \cM{\mathcal{M}}
\newcommand \cN{\mathcal{N}}
\newcommand \cP{\mathcal{P}}

\newcommand \bbA{\mathbb{A}}

\newcommand \bbP{\mathbb{P}}

\newcommand\fF{\mathbf{F}}	

\newcommand\fg{\mathfrak g}

\newcommand \fG{\mathfrak G}

\newcommand \fQ{\mathfrak Q}
\newcommand \fR{\mathfrak R}

\newcommand \fT{\mathfrak T}

\newcommand \te{{\tilde e}}
\newcommand \tf{{\tilde f}}

\newcommand \he{{\hat e}}

\newcommand \hE{{\hat E}}

\newcommand \hN{{\hat N}}

\newcommand \hz{[\infty]}
\newcommand \hZ{\langle\infty\rangle}

\newcommand\Qg{\mathfrak g}
\newcommand\Qh{\mathfrak h}
\newcommand\Qk{\mathfrak k}
\newcommand\Qd{d}
\newcommand\Qez{e_0}

\newcommand\Td{\bar d}
\newcommand\Tg{\bar{\mathfrak g}}
\newcommand\Th{\bar{\mathfrak h}}
\newcommand\Tk{\bar{\mathfrak k}}
\newcommand\Tez{\bar e_0}

\newcommand\PP{\Pi}	
\newcommand\G{{G\full}}

\newcommand\mencev{\cite[Section 2\comment{mencev}]{BG6} }		
\newcommand\ortho{\cite[Section 3\comment{ortho}]{BG6} }		
\newcommand\orthoaff{\cite[Section 3.1.2\comment{orthoaff}]{BG6} }	
\newcommand\orthopar{\cite[Section 3.1\comment{orthopar}]{BG6} }	
\newcommand\orthoaffino{\cite[Section 3.2\comment{orthoaffino}]{BG6} }	

\begin{document}
\thispagestyle{empty} 

\title{The Projective Planarity Question for \\ Matroids of $3$-Nets and Biased Graphs}

\author{Rigoberto Fl\'orez}
\address{The Citadel, Charleston, South Carolina 29409}
\email{\tt rigo.florez@citadel.edu}

\author{Thomas Zaslavsky}
\address{Dept.\ of Mathematical Sciences, Binghamton University, Binghamton, New York 13902-6000}
\email{\tt zaslav@math.binghamton.edu}

\date{\today}

\begin{abstract}
A biased graph is a graph with a class of selected circles (``cycles'', ``circuits''), called ``balanced'', such that no theta subgraph contains exactly two balanced circles.  
A biased graph has two natural matroids, the frame matroid and the lift matroid.  

A classical question in matroid theory is whether a matroid can be embedded in a projective geometry.  There is no known general answer, but for matroids of biased graphs it is possible to give algebraic criteria.  Zaslavsky has previously given such criteria for embeddability of biased-graphic matroids in Desarguesian projective spaces; in this paper we establish criteria for the remaining case, that is, embeddability in an arbitrary projective plane that is not necessarily Desarguesian.  
The criteria depend on the embeddability of a quasigroup associated to the graph into the additive or multiplicative loop of a ternary coordinate ring for the plane.  

A 3-node biased graph is equivalent to an abstract partial 3-net; thus, we have a new algebraic criterion for an abstract 3-net to be realized in a non-Desarguesian projective plane.  

We work in terms of a special kind of 3-node biased graph called a biased expansion of a triangle.  Our results apply to all finite 3-node biased graphs because, as we prove, every such biased graph is a subgraph of a finite biased expansion of a triangle.  A biased expansion of a triangle, in turn, is equivalent to an isostrophe class of quasigroups, which is equivalent to a $3$-net.  

Much is not known about embedding a quasigroup into a ternary ring, so we do not say our criteria are definitive.  
For instance, it is not even known whether there is a finite quasigroup that cannot be embedded in any finite ternary ring.  If there is, then there is a finite rank-3 matroid (of the corresponding biased expansion) that cannot be embedded in any finite projective plane---a presently unsolved problem.
\end{abstract}



\thanks{We dedicate this paper to Seyna Jo Bruskin, who during its preparation made our consultations so much more delightful.}

\thanks{Sections \ref{q frame}--\ref{q planar} are based on Chapter 2 of Fl\'orez's doctoral dissertation, supervised by Zaslavsky.}
\thanks{Fl\'orez's research was partially supported by a grant from The Citadel Foundation.  Zaslavsky's research was partially assisted by grant DMS-0070729 from the National Science Foundation.}

\maketitle


\newpage
\section*{Introduction} \label{intro}

This paper bridges a few branches of combinatorial mathematics:  matroids, graphs, projective incidence geometry, and $3$-nets.  They have been variously connected before; but here we take a corner of each and bind them into in a single system.

The central object of interest can be viewed in three ways.  Graph-theoretically, as in \cite{BG}, it is a ``biased expansion of a triangle''.  In incidence geometry, it is an abstract $3$-net.  In algebra, it is an isostrophe class of quasigroups.  (All technical terms will be precisely defined in Section \ref{prelim}.)

A biased graph is a graph with an additional structure that gives it new properties that are yet recognizably graph-like.  Notably, it has two natural generalizations of the usual graphic matroid, which we call its frame and lift matroids.  In \cite[Part IV]{BG}\footnote{To read this paper it is not necessary to know \cite{BG} or \cite{BG6}.}   we studied linear, projective, and affine geometrical representation of those matroids using coordinates in a skew field.
That leaves a gap in the representation theory because there are projective and affine geometries (lines and planes) that cannot be coordinatized by a skew field and there are biased graphs whose matroids cannot be embedded in a vector space over any skew field.  
A reason for that gap is that our representation theory depended on coordinates.  
In \cite{BG6} we close a piece of that gap with a purely synthetic alternative development of geometrical representability of the frame and lift matroids.  
In this paper, which may be considered a complement to \cite{BG6}, we develop a kind of synthetic analytic geometry (!): we produce criteria for existence of representations in non-Desarguesian planes by appealing to the fall-back coordinatization of a non-Desarguesian plane via a ternary ring, whose algebra depends on synthetic constructions (Sections \ref{planes}--\ref{questions}).  Quasigroups and their associated $3$-nets are fundamental to the planar representation theory because matroid representation of a biased expansion of the triangle graph $K_3$ (this is a kind of biased graph) is by means of ``regular'' (i.e., triangular or affine) embedding of the associated $3$-net in the plane.  
We were surprised that even with the help of the Loop \& Quasigroup Forum \cite{LQF} we could not find published criteria for a $3$-net to embed regularly in a projective plane.  Propositions \ref{L:trinet} and \ref{L:affnet} provide such criteria.  Our results are that three problems are equivalent:  representation of matroids of biased expansions of $K_3$ in a projective plane $\PP$, regular embeddability of the related $3$-net in $\PP$, and embeddability of a quasigroup of the $3$-net in the additive or multiplicative loop of a ternary ring coordinatizing $\PP$.
Furthermore, all finite biased graphs of order 3 are subgraphs of finite biased expansions of $K_3$ (Theorem \ref{T:qx}); equivalently, every finite partial 3-net extends to a finite 3-net.

Matroids of biased graphs of order 3 have been employed, usually in disguise, in several papers about matroid representability.  A brief and incomplete summary:  Reid \cite{Reid} introduced some of the ``Reid matroids'' \cite{Oxley} that were later shown to be linearly representable over only one characteristic.  Reid matroids are matroids of biased graphs of order 3; subsequently they were used by Gordon \cite{Galg}, Lindstr\"om \cite{LindAlg}, and Fl\'orez \cite{FLind} as examples of algebraic representability in a unique characteristic and of algebraic non-representability;  
by Fl\'orez \cite{FHarm} to develop the notion of harmonic matroids; by 
McNulty \cite{McN, McNnon} as examples of matroids that cannot be oriented and then by Fl\'orez and Forge \cite{FF} as minimal such matroids.  The biased graphs behind the matroids first appear explicitly in the papers of Fl\'orez.

Kalhoff \cite{K} shows that every finite rank-3 matroid embeds in some countable but possibly infinite projective plane.
It is still not known whether every finite matroid of rank 3 embeds in a finite projective plane.  By Theorems \ref{T:qx}, \ref{parte1}, and \ref{bmain}, the existence of a finite loop that is isomorphic to no subloop of the multiplicative or additive loop of any finite ternary ring implies that the related biased expansion matroid is a finitely non-embeddable rank-3 matroid (see Problems \ref{PrD:planerepb} and \ref{PrD:planerepl}).  
That leaves the problem of deciding whether a given loop does embed in a given ternary ring, or in some ternary ring, or equivalently whether a given $3$-net embeds regularly in some projective plane.  Criteria for that are not known.

We anticipate that the audience for this paper may include both matroid theorists and persons interested in projective planes and quasigroups.  We have tried to provide adequate background for both kinds of reader.  
We review the necessary background about graphs and matroids in Section \ref{prelim} and about planes, nets,  and ternary rings in Section \ref{planes}.  Besides the diverse audience, another reason for the review is the wide variation in notation and terminology about planes.  A third reason is a small innovation, the diamond operation in a ternary ring (Section \ref{ternary}), which is related to Menelaus' criterion for collinearity (Section \ref{q f point}).


\section{Graphs, biased graphs, geometry} \label{prelim}

So as not to require familiarity with previous papers on biased graphs, we state all necessary definitions.  We also provide the required background from algebra and projective geometry.

\subsection{Algebra}\label{algebra}\

We denote by $\fF$ a skew field.  Its multiplicative and additive groups are $\fF^\times$ and $\fF^+$.  

For a binary operation $*$, $*\opp$ denotes its opposite: $x *\opp y := y * x$.  

\subsubsection{Quasigroups}\label{qgps}\

A \emph{quasigroup} is a set $\fQ$ with a binary operation such that, if $x,z \in \fQ$, then there exists a unique $y \in \fQ$ such $x \cdot y = z$, and if $y,z \in \fQ$ then there exists a unique $x \in \fQ$ such that $x \cdot y = z$.    
(We write a raised dot for the operation in an abstract quasigroup to distinguish it from the operation in a subquasigroup of a ternary ring or skew field.)  
A \emph{loop} is a quasigroup with an identity element.\footnote{Fortunately, we do not use graph loops.}  
A \emph{trivial} quasigroup has only one element.  
The quasigroups $(\fQ, \cdot )$ and $(\fQ', \times)$ are \emph{isotopic} if there are bijections $\alpha$, $\beta$ and $\gamma$ from $\fQ$ to $\fQ'$ such that for every $x,y \in  \fQ$, $\alpha(x) \times \beta(y) = \gamma(x \cdot y)$ and \emph{principally isotopic} if $\gamma$ is the identity mapping (then $\fQ$ and $\fQ'$ must have the same elements).  
They are \emph{conjugate} if the equation $x \cdot y = z$ corresponds to an equation $x' \times y' = z'$ where $(x',y',z')$ is a definite permutation of $(x,y,z)$.  They are \emph{isostrophic} if one is conjugate to an isotope of the other.  
Two elementary but illuminating facts from quasigroup theory:  Every quasigroup is principally isotopic to a loop.  If some isotope of a quasigroup is a group, then every isotope that is a loop is that same group.

We have not found a source for the following simple lemma.  

\begin{lem}\label{L:subisotope}
Suppose $\fQ_1$ and $\fQ_2$ are quasigroups with isotopes $\fQ_1'$ and $\fQ_2'$ such that $\fQ_1'$ is a subquasigroup of $\fQ_2'$.  Then $\fQ_1$ and $\fQ_2$ have isotopes $\fQ_1''$ and $\fQ_2''$ that are loops with $\fQ_1''$ a subloop of $\fQ_2''$.
\end{lem}

\begin{proof}
Write out the multiplication table $T_2'$ of $\fQ_2'$ with borders $x, y$ so the element in row $x$, column $y$ is $x \cdot y$.  ($\fQ_2$ may be infinite so we do not expect anyone to carry out this instruction.)  Since $\fQ_1'$ is a subquasigroup of $\fQ_2'$, the table of $\fQ_1'$ is contained in $T_2'$.  
Choose an element $e \in \fQ_1'$ and relabel the rows and columns of $T_2'$ by the entries in, respectively, the column and row of $e$ in $T_2'$.  Labels in $\fQ_1'$ remain labels in $\fQ_1'$ because $\fQ_1'$ is closed under multiplication.  The relabelling gives a new table $T_2''$ that defines a principal isotope $\fQ_2''$ of $\fQ_2'$ that has $e$ as identity and, contained within it, a principal isotope $\fQ_1''$ of $\fQ_1'$ that also has $e$ as identity.
\end{proof}

\subsection{Projective and affine geometry}\label{geom}\

A good reference for the basics of projective and affine geometry is Dembowski \cite{Demb}.
We use the notation of Stevenson \cite{st} for planes, with a couple of obvious changes.
We assume the reader knows the axiomatic definition of a projective plane, as in \cite{Demb, HP, st} for instance.  

A $d$-dimensional projective geometry over a skew field is denoted by $\bbP^d(\fF)$.  Projective and affine geometries may have dimension $1$; that is, they may be lines (of order not less than 2); such a line is called \emph{Desarguesian} if it comes with projective (i.e., homogeneous) or affine coordinates in a skew field.  In analogy to the usual geometry, an affine line is a projective line less a point; the order of a projective line is its cardinality minus 1 and that of an affine line is its cardinality.

A plane $\PP$, on the contrary, need not be Desarguesian, but even if not, it has a coordinatization by a kind of ternary algebra $\fT$ called a ternary ring, a structure introduced by Hall \cite{PTR}.  
When the plane is Desarguesian, the ternary operation of $\fT$ is simply the formula $t(x,m,b) = xm+b$ in $\fF$.  In general, $\fT$ is neither uniquely determined by $\PP$ nor is it easy to treat algebraically.  
We treat ternary rings in Section \ref{ternary}.

\subsection{Graphs}\label{graphs}\

A graph is $\Gamma = (N,E)$, with node set $N = N(\Gamma)$ and edge set $E = E(\Gamma)$.  
Its \emph{order} is $\#N$.  Edges may be links (two distinct endpoints) or half edges (one endpoint); the notation $e_{uv}$ means a link with endpoints $u$ and $v$ and the notation $e_v$ means a half edge with endpoint $v$.\footnote{The loops and loose edges that appear in \cite{BG} are not needed here.} We allow multiple links, but not multiple half edges at the same node.  If $\Gamma$ has no half edges, it is \emph{ordinary}.  If it also has no parallel edges (edges with the same endpoints), it is \emph{simple}.  The \emph{simplification} of $\Gamma$ is the graph with node set $N(\Gamma)$ and with one edge for each class of parallel links in $\Gamma$.  The \emph{empty graph} is $\eset := (\eset,\eset)$.  
For $X\subseteq N$, $E{:}X$ denotes the set of edges whose endpoints are all in $X$.  
We make no finiteness restrictions on graphs.  

A graph is \emph{inseparable} if it has no cut node, i.e., no node that separates one edge from another.  A node incident to a half edge and another edge is a cut node; that graph is separable.

A \emph{circle} is the edge set of a simple closed path, that is, of a connected graph with degree 2 at every node.  $C_n$ denotes a circle of length $n$.  The set of circles of a graph $\Gamma$ is $\cC(\Gamma)$.  A \emph{theta graph} is the union of three internally disjoint paths with the same two endpoints.  A subgraph of $\Gamma$ \emph{spans} if it contains all the nodes of $\Gamma$.  For $S \subseteq E$, $c(S)$ denotes the number of connected components of the spanning subgraph $(N,S)$.

Our most important graph is $K_3$.  We let $N:= \{v_1, v_2, v_3 \}$ and $E(K_3):= \{e_{12}, e_{13}, e_{23} \}$ be the node and edge sets of $K_3$.

\subsection{Biased graphs and biased expansions}\label{bg}\

This exposition, derived from \cite{BG}, is specialized to graphs of order at most 3, since that is all we need for plane geometry.  Also, we assume there are no loops or loose edges.

\subsubsection{Biased graphs}\label{bgbasics}\

A \emph{biased graph} $\Omega = (\Gamma,\cB)$ consists of an underlying graph $\|\Omega\| := \Gamma$ together with a class $\cB$ of circles satisfying the condition that, in any theta subgraph, the number of circles that belong to $\cB$ is not exactly 2.  
Another biased graph, $\Omega_1$, is a \emph{subgraph} of $\Omega$ if $\| \Omega_1 \| \subseteq \| \Omega \|$ and $\cB(\Omega_1) = \cB(\Omega) \cap \cC(\Omega_1)$.

In a biased graph $\Omega$, an edge set or a subgraph is called \emph{balanced} if it has no half edges and every circle in it belongs to $\cB$.  Thus, a circle is balanced if and only if it belongs to $\cB$, and a set containing a half edge is unbalanced.  For $S \subseteq E$, $b(S)$ denotes the number of balanced components of the spanning subgraph $(N,S)$.  $N_0(S)$ denotes the set of nodes of all unbalanced components of $(N,S)$.  
A \emph{full} biased graph has a half edge at every node; if $\Omega$ is any biased graph, then $\Omega\full$ is $\Omega$ with a half edge adjoined to every node that does not already support one.
$\Omega$ is \emph{simply biased} if it has at most one half edge at each node and no balanced digons (recall that we exclude loops).  

In a biased graph there is an operator on edge sets, the \emph{balance-closure} $\bcl$,\footnote{Not ``balanced closure''; it need not be balanced.} defined by 
$$
\bcl S := S \cup \{ e \notin S : \text{ there is a balanced circle $C$ such that } e \in C \subseteq S \cup \{e\} \}
$$
for any $S \subseteq E$.  This operator is not an abstract closure since it is not idempotent, but it is idempotent when restricted to balanced edge sets; indeed, $\bcl S$ is balanced whenever $S$ is balanced \cite[Proposition I.3.1]{BG}.  We call $S$ \emph{balance-closed} if $\bcl S = S$ (note that such a set need be neither balanced nor closed).

For the general theory of biased graphs see \cite{BG}.  From now on we concentrate on order 3.

\subsubsection{Biased expansions}\label{bx3}\

A \emph{biased expansion of $K_3$} is a biased graph $\Omega$ with no half edges and no balanced digons, together with a \emph{projection mapping} $p: \|\Omega\| \to K_3$ that is surjective, is the identity on nodes, and has the property that, for the unique circle $C = e_{12} e_{23} e_{31}$ in $K_3$, each edge $e_{ij}$, and each choice of $\te \in p\inv(e)$ and $\tf$ for both edges $e,f \neq e_{ij}$, there is a unique $\te_{ij} \in p\inv(e_{ij})$ for which $\te_{ij} \te \tf$ is balanced.  
It is easy to prove that in a biased expansion of $K_3$, each $p\inv(e)$ has the same cardinality, say $\gamma$; we sometimes write $\gamma\cdot K_3$ for such an expansion.  
The \emph{trivial} biased expansion $1\cdot K_3$ consists of $K_3$ with its circle balanced.  
A \emph{full biased expansion} is $(\gamma\cdot K_3)\full$, also written $\gamma\cdot K_3\full$.

\subsubsection{Isomorphisms}\label{isom}\

An \emph{isomorphism} of biased graphs is an isomorphism of the underlying graphs that preserves balance and imbalance of circles.  
A \emph{fibered isomorphism} of biased expansions $\Omega_1$ and $\Omega_2$ of the same base graph $\Delta$ is a biased-graph isomorphism combined with an automorphism $\alpha$ of $\Delta$ under which $p_1\inv(e_{vw})$ corresponds to $p_2\inv(\alpha e_{vw})$.  In this paper \emph{all isomorphisms of biased expansions are intended to be fibered} whether explicitly said so or not.  
A \emph{stable isomorphism} of $\Omega_1$ and $\Omega_2$ is a fibered isomorphism in which $\alpha$ is the identity function (the base graph is fixed pointwise).

Similar terminology applies to monomorphisms.

\subsubsection{Quasigroup expansions}\label{qx3}\

A biased expansion of $K_3$ can be formed from a quasigroup $\fQ$.  
The \emph{$\fQ$-expansion} $\fQ K_3$ of $K_3$ consists of the underlying graph $(N, \fQ \times E_3)$, where an edge $(\Qg, e_{ij})$ (more briefly written $\Qg e_{ij}$) has endpoints $v_i,v_j $ and the set of balanced circles is 
$$
\cB(\fQ K_3) := \{ \{\Qg e_{12},  \Qh e_{23},  \Qk e_{13}\}  : \Qg \cdot \Qh = \Qk \}.
$$  
The projection mapping $p: \fQ K_3 \to K_3$ is defined by $p(v_i):=v_i$ and $p(\Qg e_{ij}) := e_{ij}.$  
We call any $\fQ K_3$ a \emph{quasigroup expansion} of $K_3$; we also loosely call the associated biased graph $\bgr{\fQ K_3} := (\|\fQ K_3\|, \cB(\fQ K_3))$ a quasigroup expansion.  

Quasigroup expansions of $K_3$ are essentially the same as biased expansions; furthermore, from $\bgr{\fQ K_3}$, one can recover $\fQ$ up to isostrophe; indeed, fibered isomorphism (or stable isomorphism) classes of expansions of $K_3$ are equivalent to isostrophe (or isotopy, respectively) classes of quasigroups.  This is shown by the next result.

\begin{prop}\label{P:isotopism}
Every quasigroup expansion of $K_3$ is a biased expansion, and every biased expansion of $K_3$ is stably isomorphic to a quasigroup expansion.  

Quasigroup expansions $\fQ K_3$ and $\fQ' K_3$ are stably isomorphic if and only if the quasigroups $\fQ$ and $\fQ'$ are isotopic, and they are fibered-isomorphic if and only if the quasigroups $\fQ$ and $\fQ'$ are isostrophic.

Isotopic quasigroups have stably isomorphic biased expansion graphs $\bgr{\fQ K_3}$ and isomorphic matroids of each kind.
\end{prop}

\begin{proof}
Clearly, a quasigroup expansion of $K_3$ is a biased expansion.  For the converse, given a biased expansion $\gamma\cdot K_3$, choose any set $\fQ$ with $\gamma$ elements, choose any bijections $\beta_{12}, \beta_{13}, \beta_{23}:  \fQ \to p\inv(e_{12}), p\inv(e_{13}), p\inv(e_{23})$, and for $\Qg, \Qh \in \fQ$ define $\Qg \cdot \Qh := \beta_{13}\inv(\te_{13})$ where $\te_{13}$ is the unique edge in $p\inv(e_{13})$ such that $\{\beta_{12}(\Qg),\beta_{23}(\Qh), \te_{13}\}$ is balanced.  It is easy to see from the definition of a biased expansion that this is a quasigroup operation.

Stable isomorphism follows because the quasigroup multiplication is determined by the balanced triangles of $\fQ K_3$ and (accounting for isotopy) the choice of bijections.  

The property of fibered isomorphism is similar except that a fibered isomorphism may involve an automorphism of $K_3$, thus permitting the quasigroups to differ by conjugation as well as isotopy.

In the last paragraph, the definition of balanced triangles implies the biased expansions are isomorphic if edges are made to correspond according to the isotopism.  The matroids of each kind (frame, full frame, lift, and extended lift) are isomorphic because the biased expansion graphs are.
\end{proof}

\subsubsection{Gain graphs and group expansions}\label{gexp}\

When the quasigroup is a group $\fG$ there are additional properties.  Orient $K_3$ arbitrarily and orient $\fg e$ similarly to $e$.  The \emph{gain} of $\fg e$ in the chosen direction is $\fg$ and in the opposite direction is $\fg\inv$.  The mapping $\phi: E(\fG K_3) \to \fG: \fg e \mapsto \fg$ or $\fg\inv$, depending on the direction, is the \emph{gain function} of $\fG K_3$.  We disambiguate the value of $\phi(e)$ on an edge $e_{ij}$, when necessary, by writing $\phi(e_{ij})$ for the gain in the direction from $v_i$ to $v_j$.  Note that $\phi$ is not defined on half edges.  
Also note that, since group elements are invertible, we can define balanced triangles in $\fG K_3$ as edges such that $\phi(e_{12})\phi(e_{23})\phi(e_{31}) =$ the group identity.  

We call $\fG K_3$ a \emph{group expansion} of $K_3$.
Group expansions of any graph can be defined in a similar way (cf.\ \cite[Section I.5]{BG}), but we only need them for subgraphs of $K_3$.  

A \emph{gain graph} $\Phi = (\Gamma,\phi)$ is any subgraph of a group expansion, with the restricted gain function, possibly together with half edges (but the gain function is defined only on links).  

\subsubsection{Biased graphs of order $3$}\label{bg3}\

Order 3 is special in more than the existence of a quasigroup for every biased expansion of $K_3$.  For order 3, but not higher orders, every finite biased graph is contained in a finite biased expansion.  (We also expect an infinite biased graph of order 3 to be contained in a triangular biased expansion of the same cardinality but we have not tried to prove it.)  

\begin{thm}\label{T:qx}
Every finite biased graph $\Omega=(\Gamma,\cB)$ of order $3$ is a subgraph of a finite biased expansion of $K_3$.  
\end{thm}

\begin{proof}
We may assume $\Omega$ has no loose or half edges or balanced digons.  The simplification of $\Gamma$ is a subgraph of $K_3$ so its edge set is $A \cup B \cup C$ where $A=\{a_1,\ldots,a_p\}$ is the set of edges with endpoints  $v_1,v_2$, $B=\{b_1,\ldots,b_q\}$ is the set of edges with endpoints $v_2,v_3$, and $C=\{c_1,\ldots,c_r\}$ consists of the edges with nodes $v_1,v_3$.  We assume $r \leq q \leq p$.  We call a biased graph that contains $\Omega$ and has the same node set a \emph{thickening} of $\Omega$.

We show how to represent $\Omega$ as a simple bipartite graph $G$ that is the union of $r$ partial matchings.  The two node classes of $G$ are $A$ and $B$.  There is a partial matching $M_k$ for each edge $c_k$; $M_k$ consists of an edge $a_ib_j$ for each balanced triangle $a_ib_jc_k$.  The size of $M_k$ is the number $\tau_k$ of balanced triangles on $c_k$.  If no balanced triangle contains $c_k$, then $M_k=\eset$.  The \emph{matching structure} $\hat G:=(G;M_1,\ldots,M_r)$ is sufficient to determine $\Omega$, so the biased graph and the matching structure are equivalent concepts.  

The biased graph is a biased expansion if and only if every $M_k$ is a complete matching.  
If $\Omega$ is contained in a supergraph $\Omega'$ (also of order 3), then $G \subseteq G'$, $r \leq r'$, and $M_k \subseteq M'_k$ for $1\leq k \leq r$, with the additional condition that any edge of $M'_k$ (for $k\leq r$) whose endpoints are in $G$ must belong to $M_k$; we call $\hat G'$ that satisfies these requirements an \emph{extension} of $\hat G$.  
Thus, thickening $\Omega$ to a biased expansion $\gamma\cdot K_3$ is equivalent to extending $\hat G$ to a (usually larger) complete bipartite graph with a 1-factorization, amongst whose 1-factors are $r$ perfect matchings that extend those of $\hat G$ subject to the extension condition just stated.

We give a relatively simple constructive proof for $r>0$ that makes the biased expansion larger than necessary.  (The case $r=0$ is treated at the end.)  Let $\hat G^0:=\hat G$, so $M_k^0:=M_k$.  We construct $\hat G^1$.  For each node of $G^0$ that is unmatched in $M_1^0$, add a new node in the opposite node class with an edge joining the two; the new edge belongs to $M_1^1$, which also contains all the edges of $M_1^0$.  This adds $p+q-2\tau_1$ new nodes, resulting in $G^1$ with $p+q-\tau_1$ nodes in each node class and with $p+q-\tau_1$ new edges, all belonging to $M_1^1$.  Let $M_k^1:=M_k^0$ for $1 < k \leq r$.  We now have $\hat G^1$.  In it, the edges of $M_1^1$ form a complete matching.  (If $M_1^0$ was already a complete matching, $\hat G^1=\hat G^0$.)  The number of edges now required to make $M_k^1$ a complete matching is $(p+q-\tau_1)-\#M_k$.

The further steps, constructing $\hat G^i$ from $\hat G^{i-1}$ for $i>1$, are slightly different because we wish to preserve the completeness of any existing complete matching.  We enlarge $N(G^{i-1})$ by adding $\max(r,(p+q-\tau_1)-\#M_k)$ nodes to each class.  The sets of new nodes are $A^i$ and $B^i$.  We modify $M_i^{i-1}$ the same way we modified $M_1^0$ except that if not all new nodes are matched in $M_i^i$ we add edges to match the unmatched new nodes; and also each $M_k^{i-1}$ for $k \neq i$ is enlarged by the edges of a complete matching of the new nodes.  
Since there are at least $r$ new nodes in each class, there exist $r$ edge-disjoint complete matchings of the new nodes.  (Proof:  Extend the partial matching $M_i^i$ of the new nodes to a complete matching $M$; use K\"onig's theorem to decompose $K_{A^i,B^i}\setminus M$ into $\#A^i \geq r$ 1-factors; and use $r-1$ of these 1-factors as the additional edges of the $r-1$ matchings $M_k^i$.)

The number of nodes in each class unmatched by an $M_k^i$ for $k>i$ is the same as the number unmatched by $M_k^{i-1}$, that is, $(p+q-\tau_1)-\#M_k$, and if $k \leq i$ there are no unmatched nodes.
Thus, at each step one more $M_i$ becomes a complete matching of the current set of nodes; in the end we have $\gamma:=\#A'$ nodes in each node class, where $A' := A \cup A^1 \cup \cdots \cup A^m$, and $r$ edge-disjoint complete matchings between $A'$ and $B' := B \cup B^1 \cup \cdots \cup B^m$.  Furthermore, we never added any edges within the original node set $A \cup B$.

We appeal to K\"onig's theorem once more: the graph $K_{A',B'} \setminus (M_1^r \cup \cdots \cup M_r^r)$ is a regular bipartite graph of degree $\gamma-r$.  It therefore has a 1-factorization into $\gamma-r$ complete matchings $M'_{r+1},\ldots,M'_{\gamma}$.  Let $M_k':=M_k^r$ for $k \leq r$.  The matching system $(K_{A',B'}; M_1',\ldots,M_\gamma')$ extends $\hat G$ and corresponds to a biased expansion $\gamma\cdot K_3$ that thickens $\Omega$.  That concludes the case where $r>0$.

In case $r=0$, we modify $\hat G$ to $\hat G'$ by adding $p-q$ nodes to $B$, so $A':=A$ and $\#A'=\#B'=p$, and taking $M'_1,\ldots,M'_p$ to be the 1-factors of a 1-factorization of $K_{A',B'} \cong K_{p,p}$.
\end{proof}

For the convenience of readers familiar with nets, we state the $3$-net version of this theorem in Corollary \ref{C:p3net}.

We can bound $\gamma$ loosely.  The number of nodes in each class after the first step is $p+q-\tau_1$.  At each following step $i$ the number of nodes in each class increases by $\max(r, p+q-\tau_1-\tau_i) \leq p+q-\tau_i$ (since $\tau_i \leq \min(p,q)=q$), so $p+q-\tau_i \geq p \geq r$.  Therefore, $\gamma \leq \sum_{i=1}^r p+q-\tau_i = r(p+q)-\tau$ where $\tau$ is the total number of balanced triangles.  In terms of edges, $\#E(\gamma\cdot K_3) < \frac13 (\#E(\Omega))^2$.  This is certainly a crude bound but not that large; e.g., it is not exponential in $\#E(\Omega)$.  A more careful analysis might lead to a construction that yields a much closer bound on $\gamma$ in terms of $p, q, r, \tau$, and possibly the $\tau_i$.

In the language of $3$-nets (Section \ref{nets}) Theorem \ref{T:qx} says that every finite incidence system $(\cP,\cL)$ of points and lines (lines are intrinsic objects, not sets of points), with the lines divided into 3 classes, such that every point is incident with three lines and two lines of the same class have no common point (a common point of two lines is a point that is incident with both), is a subsystem of a finite abstract 3-net $\cN = (\cP(\cN),\cL(\cN))$, in the sense that $\cP \subseteq \cP(\cN)$ and $\cL \subseteq \cL(\cN)$, with incidence in $(\cP,\cL)$ implying incidence in $\cN$.

\subsection{Matroids}\label{matroids}\

We assume acquaintance with elementary matroid theory as in \cite[Chapter 1]{Oxley}.  
The matroid of projective (or affine) dependence of a projective (or affine) point set $A$ is denoted by $M(A)$.  
All the matroids in this paper are finitary, which means that any dependent set contains a finite dependent set.
They are also simple, which means every circuit has size at least 3.

\comment{MAKE CIRCUIT FIGURES.}

In an ordinary graph $\Gamma$ an edge set $S$ is \emph{closed} in $\Gamma$ if whenever $S$ contains a path joining the endpoints of an edge $e$, then $e \in S$; this closure defines the graphic matroid $G(\Gamma)$.  
The graphic matroid $G(\Gamma)$ can be defined by its circuits, which are the circles of $\Gamma$.  
The rank function of $G(\Gamma)$ is $\rk S = \#N - c(S)$ for an edge set $S$.  

A biased graph $\Omega$, on the other hand, gives rise to two kinds of matroid, both of which generalize the graphic matroid.  (The matroids of a gain graph $\Phi$ are those of its biased graph $\bgr{\Phi}$.)  

First, the \emph{frame matroid}, written $G(\Omega)$, has for ground set $E(\Omega)$.  A \emph{frame circuit} is a circuit of $G(\Omega)$; it is a balanced circle, a theta graph that contains no balanced circle, or two unbalanced figures connected either at a single node or by a simple path.  (An \emph {unbalanced figure} is a half edge or unbalanced circle.)  
For example, if $\Omega$ is a simply biased graph of order 3, the frame circuits are as shown in Figure \ref{F:framecirc3}.
The rank function in $G(\Omega)$ is $\rk S = \#N - b(S)$; thus, $G(\Omega)$ has rank 3 if $\Omega$ has order 3 and is connected and unbalanced.  
The \emph {full frame matroid} of $\Omega$ is $\G(\Omega) := G(\Omega\full)$ (the frame is the set of half edges at the nodes in $\Omega\full$).  $\G(\Omega)$ has rank 3 if (and only if) $\Omega$ has order 3.
The archetypic frame matroid is the Dowling geometry of a group \cite{CGL}, which is $\G(\fG K_n)$.  (Dowling also mentions $\G(\fQ K_3)$, where $\fQ$ is a quasigroup, and shows that there is no exact higher-order analog.  The nature of the nearest such analogs is established in \cite{AMQ}.)

\begin{figure}[htbp]
\begin{center}
\begin{tabular}{ccc}
\includegraphics[scale=.25]{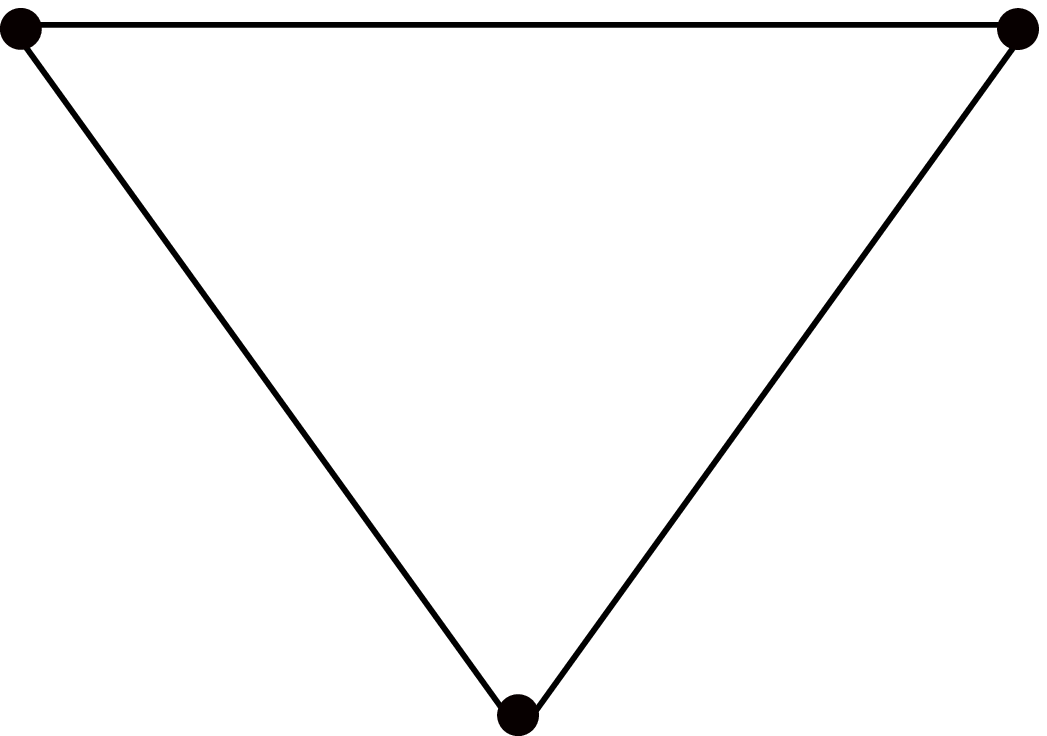} \hspace{1cm} 
&
\includegraphics[scale=.25]{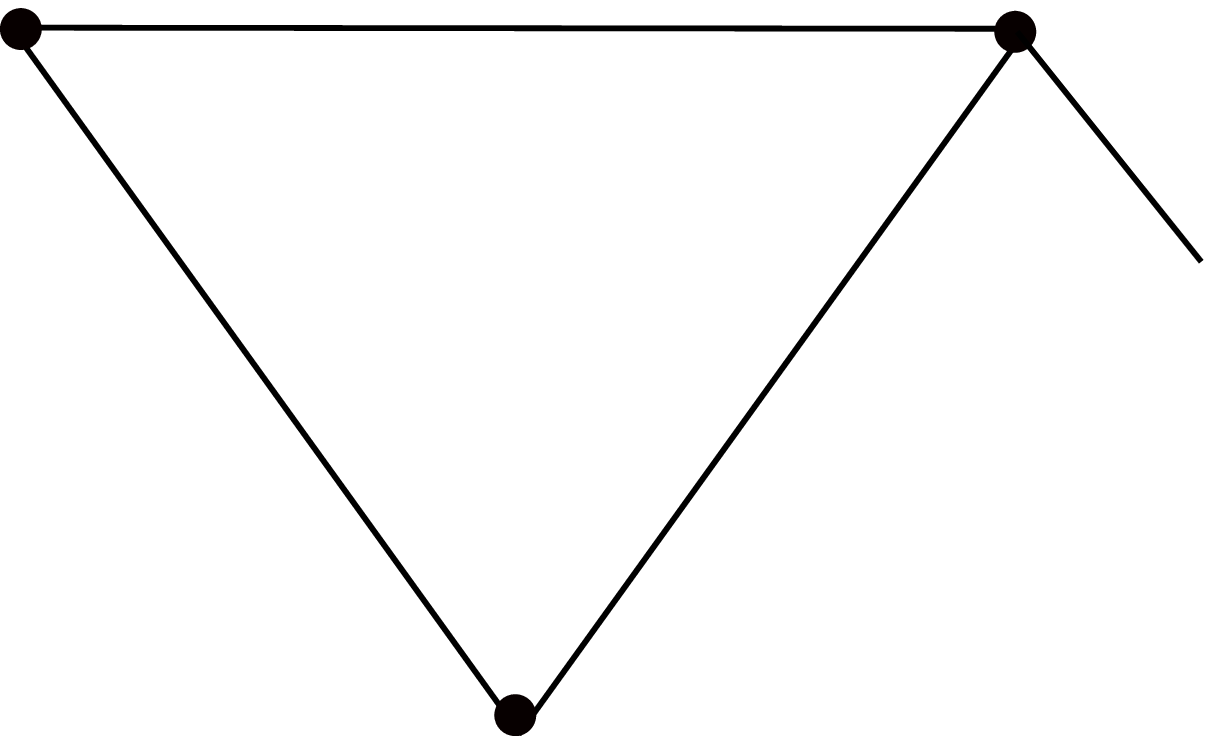} \hspace{1cm} 
&
\includegraphics[scale=.25]{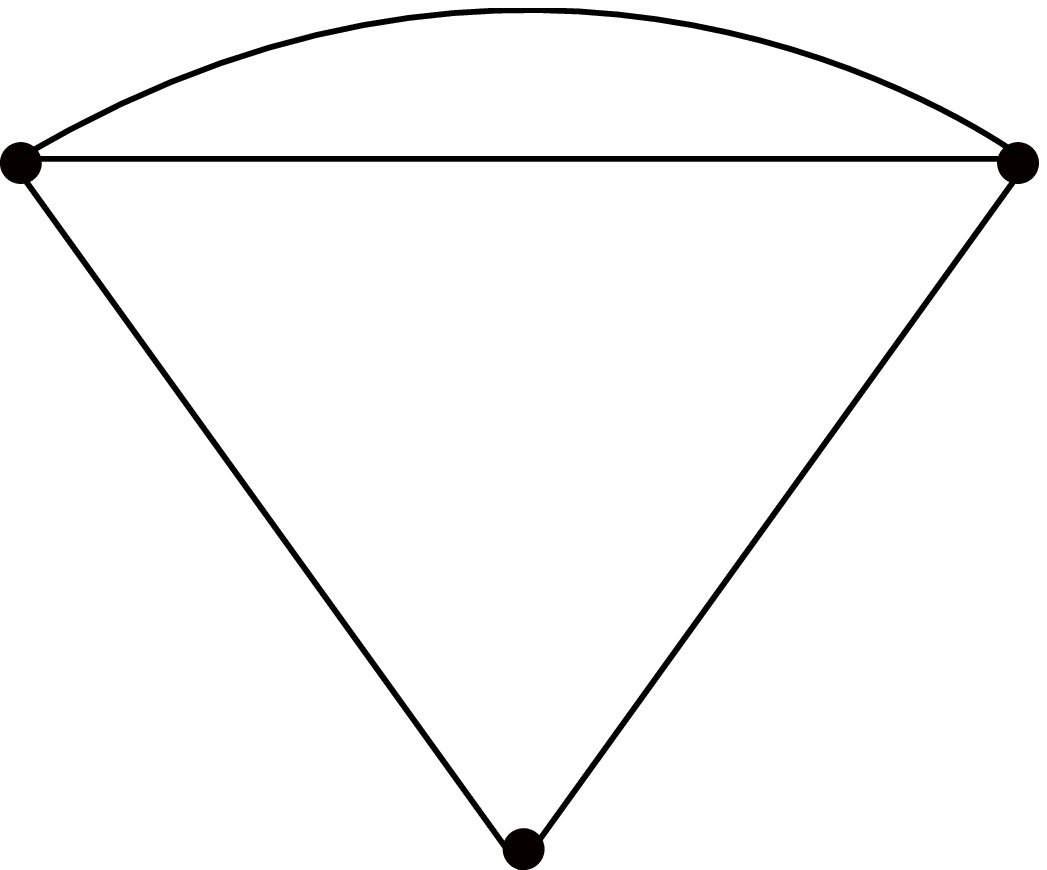} 
\\[12pt]
\includegraphics[scale=.25]{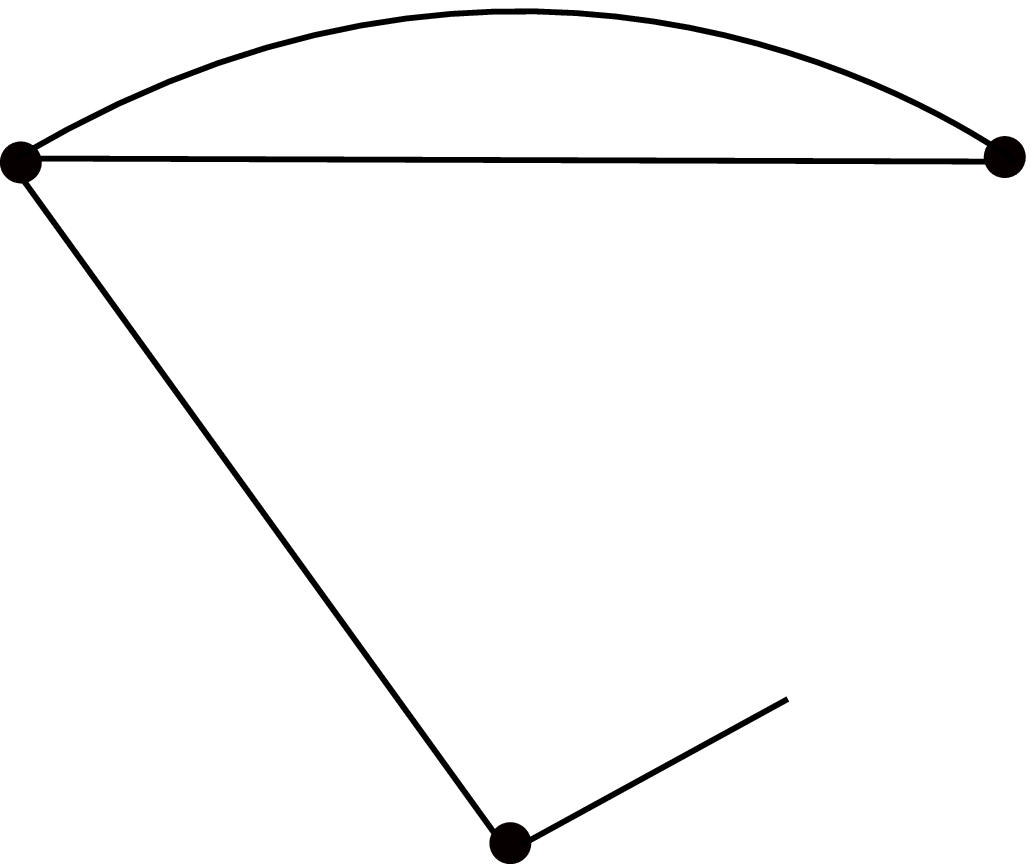} \hspace{1cm} 
&
\includegraphics[scale=.25]{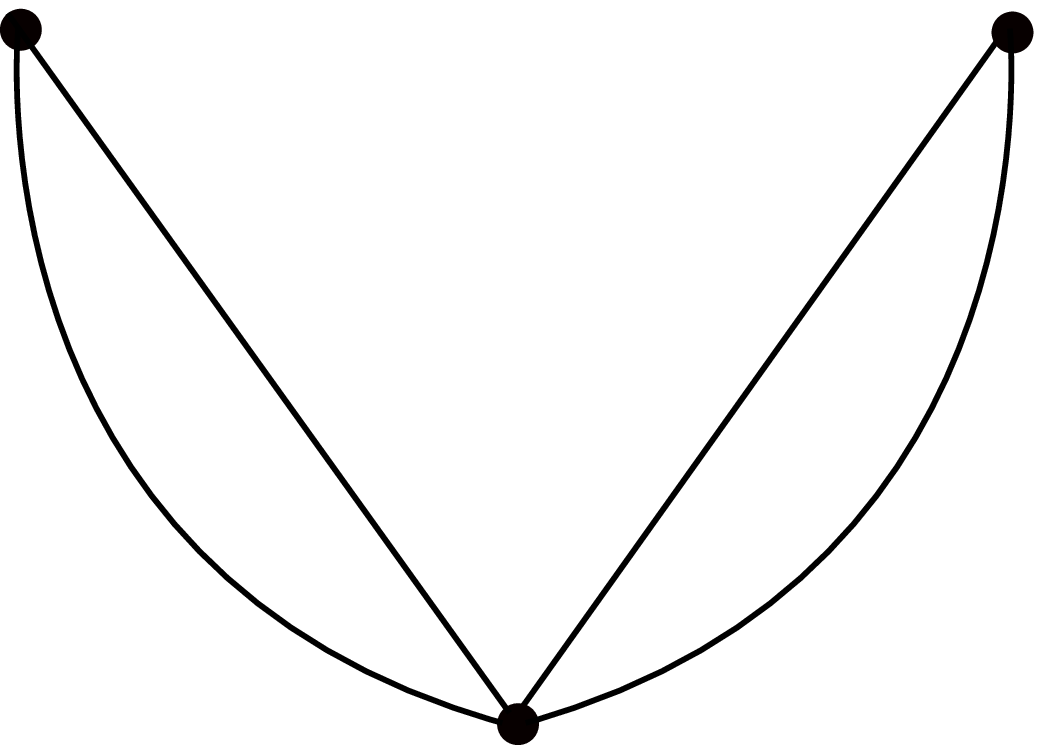} \hspace{1cm} 
&
\raisebox{30pt}{\includegraphics[scale=.25]{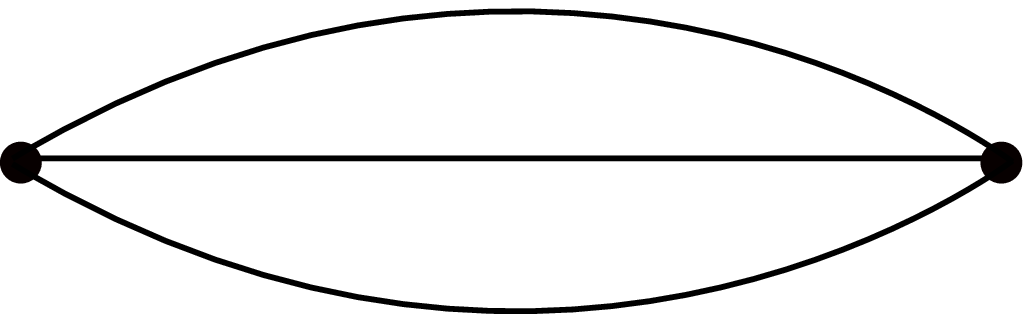}} 
\\[15pt]
\raisebox{30pt}{\includegraphics[scale=.25]{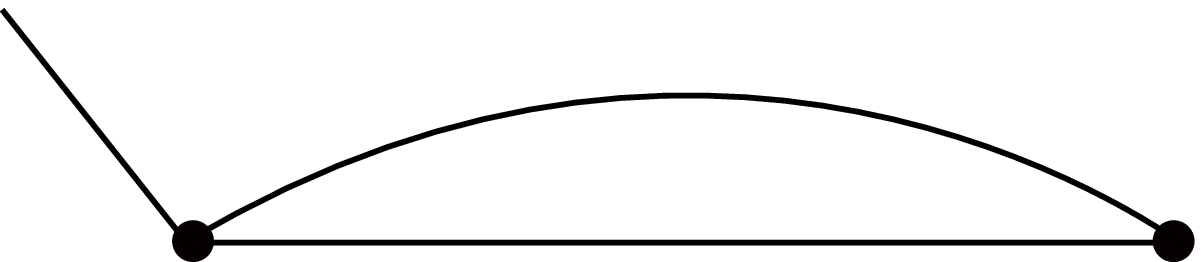}} \hspace{1cm} 
&
\raisebox{12pt}{\includegraphics[scale=.25]{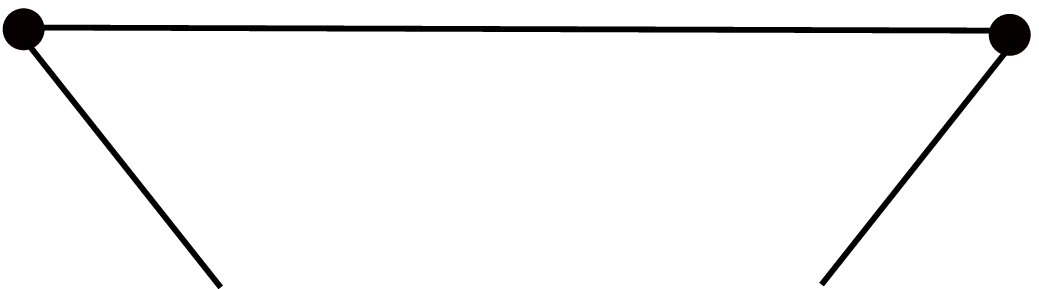}} \hspace{1cm} 
&
\includegraphics[scale=.25]{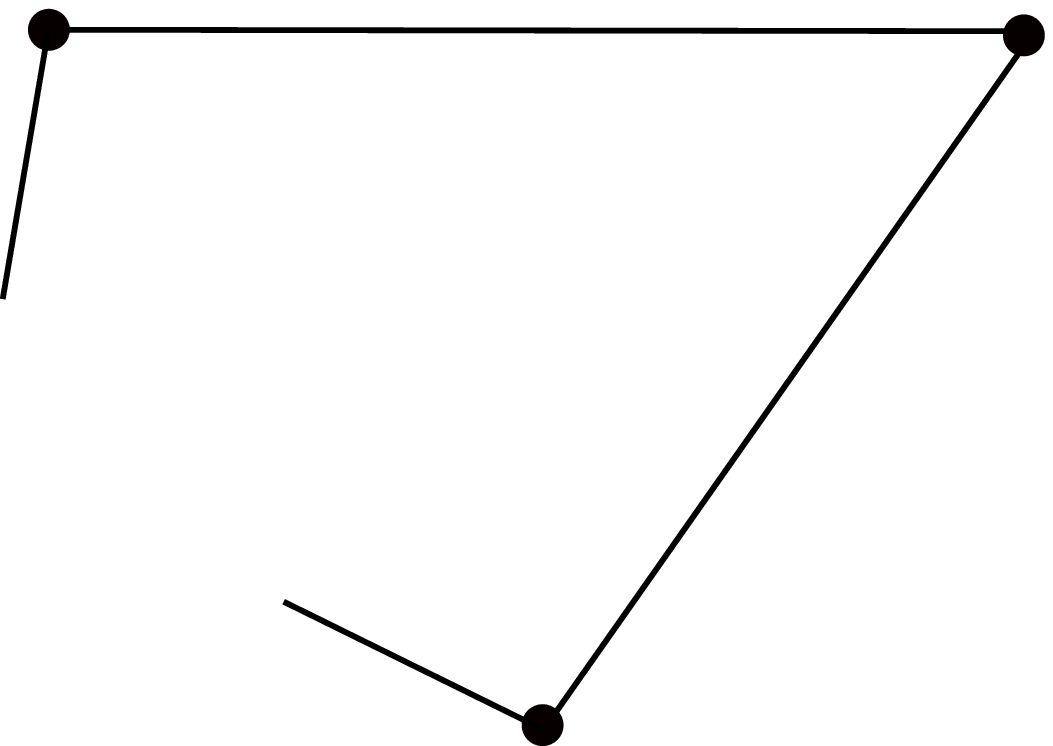} 
\end{tabular}
\bigskip
\caption{The frame circuits that may exist in a simply biased graph of order 3, such as a full biased expansion of $K_3$.  The circles are unbalanced except for the triangle in the upper left.}
\label{F:framecirc3}
\end{center}
\end{figure}

Second, the \emph{extended lift matroid} $L_0(\Omega)$ (also called the \emph {complete lift matroid}), whose ground set is $E(\Omega) \cup \{e_0\}$, that is, $E(\Omega)$ with an extra element $e_0$.  A \emph{lift circuit}, which is a circuit of $L_0(\Omega)$, is either a balanced circle, a theta graph that has no balanced circle, two unbalanced figures connected at one node, two unbalanced figures without common nodes, or an unbalanced figure and $e_0$.  
The lift circuits of a simply biased graph of order 3 are shown in Figure \ref{F:liftcirc3}.  
The \emph{lift matroid} $L(\Omega)$ is $L_0(\Omega) \setminus e_0$.  The rank function in $L_0(\Omega)$, for $S \subseteq E(\Omega) \cup \{e_0\}$, is $\rk S = \#N - c(S)$ if $S$ is balanced and does not contain $e_0$, and $= \#N - c(S) + 1$ if $S$ is unbalanced or contains $e_0$.  
Thus, $L(\Omega)$ has rank 3 if $\Omega$ has order 3 and is connected and unbalanced; $L_0(\Omega)$ has rank 3 if $\Omega$ has order 3 and is connected.

\begin{figure}[htbp]
\begin{center}
\begin{tabular}{ccc}
\includegraphics[scale=.25]{Biased_Graph_1.eps} \hspace{1cm} 
&
\includegraphics[scale=.25]{Biased_Graph_2.eps} \hspace{1cm} 
&
\includegraphics[scale=.25]{Biased_Graph_3.eps} 
\\[12pt]
\includegraphics[scale=.25]{Biased_Graph_5.eps} \hspace{1cm} 
&
\includegraphics[scale=.25]{Biased_Graph_7.eps} \hspace{1cm} 
&
\raisebox{30pt}{\includegraphics[scale=.25]{Biased_Graph_6.eps}} 
\\[15pt]
\raisebox{15pt}{\includegraphics[scale=.25]{Biased_Graph_4.eps}} \hspace{1cm} 
&
\includegraphics[scale=.25]{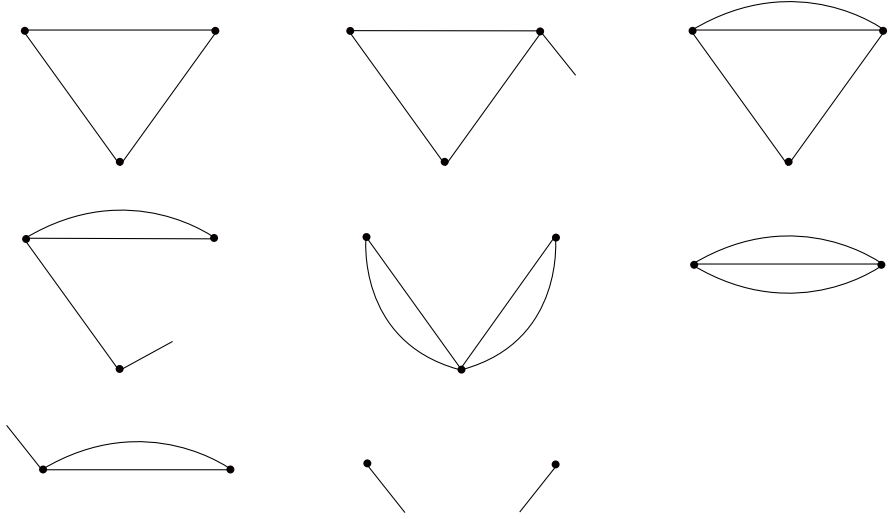} \hspace{1cm} 
\end{tabular}
\bigskip
\caption{The lift circuits that may exist in a simply biased graph of order 3.  The circles are unbalanced except for the triangle the upper left.}
\label{F:liftcirc3}
\end{center}
\end{figure}

When $\Omega$ is balanced, $G(\Omega)=L(\Omega)=G(\|\Omega\|)$, the graphic matroid, and $L_0(\Omega)$ is isomorphic to $G(\|\Omega\| \cup K_2)$ (a disjoint or one-point union).
When $\Omega$ has $\#N\leq3$, is simply biased, and has no half edges, then $G(\Omega)=L(\Omega)$; that is, there is only one matroid. By contrast, even with those assumptions normally $\G(\Omega)\neq L_0(\Omega)$.

A (vector or projective) representation of the frame or lift matroid of $\Omega$ is called, respectively, a \emph{frame representation} (a ``bias representation'' in \cite[Part IV]{BG}) or a \emph{lift representation} of $\Omega$.  An \emph{embedding} is a representation that is injective; as all the matroids in this paper are simple, ``embedding'' is merely a short synonym for ``representation''.  
A frame or lift representation is \emph{canonical} if it extends to a representation of $\G(\Omega)$ (if a frame representation) or $L_0(\Omega)$ (if a lift representation).  A canonical representation, in suitable coordinates, has especially simple representing vectors \cite[Sections IV.2.1 and IV.4.1]{BG} that we shall explain as needed.  

Like projective planes among projective geometries, biased expansions of $K_3$ are peculiar among biased expansions.  Since $\Omega=\gamma\cdot K_3$ has no node-disjoint circles, its frame and lift matroids are identical; consequently, its frame and lift representations are also identical.  Nevertheless, $\Omega$ has two different kinds of projective representation (when $\gamma>1$) because its canonical frame and lift representations are not identical.  

\begin{prop}\label{L:canonical}
Every projective representation of a nontrivial biased expansion $\Omega$ of $K_3$ (that is, of the matroid $G(\Omega) = L(\Omega)$) extends to a representation of either $L_0(\Omega)$ or $\G(\Omega)$, but not both.  
\end{prop}

Proposition \ref{L:canonical} is obvious.  Three lines in a plane are concurrent in a point, or not.  
The former case, the extended lift representation, occurs when the edge lines are concurrent, and the latter, the full frame representation, otherwise.  (The \emph{edge lines} are the lines determined by the points representing the three edge sets $p\inv(e_{ij})$.)  The reader may recognize this as similar to $3$-nets; in fact, the former case is dual to an affine $3$-net and the latter to a triangular $3$-net.  

The next result explains the significance of canonical representation and the importance of a biased graph's having only canonical representations.  It gives algebraic criteria for representability of matroids of biased expansion graphs in Desarguesian projective geometries.  This paper develops analogs for non-Desarguesian planes.

\begin{thm}\label{T:canonicalrep}
Let $\Omega$ be a biased graph and $\fF$ a skew field.  
\begin{enumerate}[{\rm(i)}]
\item The frame matroid $G(\Omega)$ has a canonical representation in a projective space over $\fF$ if and only if $\Omega$ has gains in the group $\fF^\times$.  
\item The lift matroid $L(\Omega)$ has a canonical representation in a projective space over $\fF$ if and only if $\Omega$ has gains in the group $\fF^+$.  
\end{enumerate}
\end{thm}

\begin{proof}
The first part is a combination of \cite[Theorem IV.2.1 and Proposition IV.2.4]{BG}.  The second part is a combination of \cite[Theorem IV.4.1 and Proposition IV.4.3]{BG}. 
\end{proof}


\section{Planes, nets, and their matroids}\label{planes}

\subsection{Quasigroup matroids and their planar representations}\label{qmatroids}\

Consider the representation problem for a quasigroup expansion $\fQ K_3$.  If the quasigroup $\fQ$ is a multiplicative subgroup of a skew field $\fF$, then the full frame matroid of this expansion graph is linearly representable over $\fF$ (\cite{CGL} or \cite[Theorem IV.2.1]{BG}).  Similarly, if $\fQ$  is an additive subgroup of $\fF$, then the extended lift matroid of $\fQ K_3$ is linearly representable over $\fF$ \cite[Theorem IV.4.1]{BG}.  In each case, the matroid embeds in a projective space coordinatizable over $\fF$, so $\fQ$'s being a subgroup of $\fF^\times$ or $\fF^+$ is a sufficient condition for representability of the full frame matroid or the extended lift matroid, respectively, in a Desarguesian projective space.  It is also a necessary condition (if $\fQ$ is a loop, which is achievable by isotopy); that was proved by Dowling \cite{CGL} for the multiplicative case and follows from \cite[Theorem IV.7.1]{BG} in both cases.

The natural question for expansions of $K_3$, whose full frame and extended lift matroids have rank 3, is the same: necessary and sufficient conditions for when the matroid is representable in a projective plane.  But since vector representability is not general enough for projective-planar representability, we need non-Desarguesian analogs of the algebraic coordinatizability criteria we know for the Desarguesian case; that means criteria for projective representability in terms of a ternary ring associated to the projective plane (cf.\ Section \ref{ternary}).

\subsection{Ternary rings}\label{ternary}\

A ternary ring is the algebraic structure that is necessary and sufficient for coordinatization of a projective plane.  It is known that given a projective plane there is a ternary ring that coordinatizes it and vice versa.  Let $\fR$ be a set with a ternary operation $t$.  
Then $\fT = (\fR,t)$ is a \emph{ternary ring} (\cite{st}; also ``ternary field'' \cite{Demb, Hall}, ``planar ternary ring'' \cite{HP}, ``Hall ternary ring'' \cite{Martin}) if it satisfies T\ref{TRx}--T\ref{TR01}.  
\begin{enumerate}[\ T1.]
\item \  Given $a,b,c,d \in \fR$ such that $a \neq c$, there exists a unique $x \in \fR$ such that $t(x,a,b)=t(x,c,d).$
\label{TRx}
\item \  Given $a,b,c \in \fR,$ there exists a unique $x \in \fR$ such that $t(a,b,x) = c.$
\label{TRb}
\item \  Given $a,b,c,d \in \fR$ such that $a \neq c$, there exists a unique pair $(x,y) \in \fR \times \fR$ such that $t(a,x,y) = b$ and $t(c,x,y)=d.$
\label{TRmb}
\item \  There exist elements $0,1 \in \fR$ such that $ 0 \neq 1$, $t(0,a,b)=t(a,0,b)=b,$ and $t(1,a,0)=t(a,1,0)=a.$
\label{TR01}
\end{enumerate}
See any of \cite{Demb, Hall, HP, st} for exposition and proofs about ternary rings.

The \emph{additive loop} of a ternary ring $\fT$ is $\fT^+ := (\fR,+)$ where $a+b:= t(1,a,b)$; it is a loop with identity $0$.\footnote{This is Stevenson's definition \cite{st}.  Hall \cite{Hall} and Dembowski \cite{Demb} define $a+b:=t(a,1,b)$.  There is some reason to think their definition could simplify some of our results; in particular, it might change $\fT^*$ to $\fT$.  We have not explored the effect of changing the definition.}  
The \emph{multiplicative loop} of $\fT$ is $\fT^\times := (\fR \setminus {0},\times)$ where $a \times b:= t(a,b,0)$; it is a loop with identity $1$. 
A ternary ring is \emph{linear} if $t(a,b,c) = (a \times b)+c$.

Let $\diamond$ be the binary operation over $\fR\setminus {0}$ defined by $x \diamond y := z$ for $z \in \fR \setminus {0}$ such that $t(x,y,z) =0$.  Thus, $\diamond = (\times^*)\opp$, i.e., $x\times y + y\times^*x = 0$.  Define 
$$
\fT^{\diamond} = (\fR\setminus {0}, \diamond).
$$
Then $\fT^{\diamond} = (\fT^*)^\times$ so $\fT^\diamond$ is a quasigroup.  Sometimes it is not new.

\begin{prop} \label{P:Tlinear}
If $\fT$ is a linear ternary ring, $\fT^{\diamond}$ and $\fT^\times$ are isotopic.
\end{prop}

\begin{proof} 
The isotopisms $\alpha, \beta : \fT^{\diamond} \to \fT^\times$ are identity mappings but $\gamma(x) := y$ for $y \in \fR $ such that $y+x=0$.  $\gamma$ is a bijection between $\fR$ and itself with $\gamma(0) = 0$ because $\fT^+$ is a loop with identity $0$.  

Suppose that $x \diamond y = z$, that is, $t(x,y,z) = 0$. Since $\fT$ is linear, $(x\times y)+z = 0$. This implies that $\gamma(x \diamond y) = x\times y$.
\end{proof}

The \emph{dual ternary ring} $\fT^*$ has the same set as $\fT$ with ternary operation $t^*$ defined by $t^*(a,b,c)=d \iff t(b,a,d)=c$.  The 0 and 1 are the same as in $\fT$.  The dual binary operations are $+^*$ and $\times^*$, defined by $x +^* y := t^*(1,x,y)$ and $x \times^* y := t^*(x,y,0)$.  
Since $x+^*y = z \iff x+z = y$, dual addition is in effect a reverse primal subtraction.  
The primal and dual multiplications are related by $t^*(x,y,y\times x) = 0$; in terms of the diamond operation, $x\times^*y = y\diamond x$, or 
\begin{equation}
\times^* = \diamond\opp.
\label{E:diamonddual}
\end{equation}
If $\fT$ is linear, $(y\times x) + (x\times^*y) = 0$.

\subsection{Projective planes}\label{pp}\

For every ternary ring $\fT$, there is an associated projective plane, and vice versa.  (See \cite[Section 8.3]{Latin}, \cite[Section V.3]{HP}, or \cite[Section 9.2, esp.\ Theorem 9.2.2 and Construction 9.2.4]{st}, from which we take our presentation.)  

The plane associated to $\fT$ is the projective plane $\PP_\fT=(\cP,\cL)$ defined by   
\begin{align*}
\cP \ = \ & \big\{ [x, y], [x], \hz: x,y \in \fR \big\}, \text{ the set of \emph{points},} \\ 
\cL \ = \ & \big\{ \langle x ,y  \rangle,  \langle x  \rangle,  \hZ: x,y \in \fR \big \},  \text{ the set of \emph{lines}, }
\end{align*}
with collinearity given by
\begin{equation}
\begin{aligned}{}
[x, y] &\in \langle m,k  \rangle \iff t(x,m,k) = y , \\ 
[x] &\in \langle x,y  \rangle , \\ 
[x,y], \hz &\in \langle x \rangle , \\ 
[x], \hz &\in \hZ , 
\end{aligned}
\label{E:coordinates}
\end{equation}
for $x,y,m,k \in \fR$.  The plane is said to be \emph{coordinatized by $\fT$}.  We call the point and line coordinates just defined the \emph{affine coordinates} in $\PP$ to distinguish them, in Sections \ref{q frame} and \ref{q lift}, from homogeneous coordinates over a skew field.

Conversely, given a projective plane $\PP$, one constructs a coordinatizing ternary ring $\fT(u,v,o,e) = (\fR,t)$ by choosing a quadruple of distinct points $u,v,o,e\in\PP$, no three collinear, taking $\fR$ to be a set in bijection with $ou \setminus \{u\}$, and defining $t(a,b,c)$ to correspond with a certain point on $ou$ determined by $a,b,c$ (cf.\ \cite[Construction 9.2.4]{st}, where $\fR$ is taken to be $ou \setminus \{u\}$, but that is not essential).  The principal properties of $\fT(u,v,o,e)$ are stated in Theorem \ref{pst}.  
(From now on we drop the notation $\fR$ and write $\fT$ for the underlying set of a ternary ring.)

\begin{thm}[{\cite[Section 9.2]{st}}] \label{pst} 
\begin{enumerate} [{\rm(a)}]
\item Let $u,v,o,e$ be points in a projective plane $\PP$, no three collinear. Then $\fT(u,v,o,e)$ is a ternary ring, in which $0 \lra o$ and $1 \lra i:= ve \wedge ou$.  The plane $\PP_{\fT(u,v,o,e)}$ is isomorphic to $\PP$.   The points in $ou\setminus \{u\}$, $ov\setminus \{v\}$, $iv\setminus \{ v \} $, and $uv\setminus \{v\}$ have, respectively, the coordinates $[x,0]$, $[0,x]$, $[1,x]$, and $[x]$ for $x \in \fT$.  In particular, $u,v,o,e,i$ have coordinates $[0], \hz, [0,0], [1,1], [1,0]$.  Some line coordinates are $ou = \langle0,0\rangle$, $ov = \langle0\rangle$, $ve = \langle1\rangle$, and $uv = \hZ$.
\label{pst:ptp}
\item Let $\fT$ be a ternary ring and $\PP = \PP_\fT$.  
If we define $u:=[0]$, $e:=[1,1]$, $v:=\hz$, $o:=[0,0]$, and in general $\fT\lra ou\setminus\{u\}$ by $x \lra [x,0]$, then $\fT(u,v,o,e) = \fT$.
\label{pst:tpt}
\end{enumerate}
\end{thm}

Note that in general $\fT(u,v,o,e)$ depends on the choice of $(u,v,o,e)$.

\begin{lem} \label{collinearpointsinaplane}
Let $\fT$ be a ternary ring and $w,x,y,m \in \fT$.  The points $[0, w], [x, y], [m] \in \PP_\fT$ are collinear if and only if $t(x, m, w) = y$.  Their common line is $\langle m, w \rangle$.  In particular, 
$$
[ \Tg],\ [ 0, \Th], \text{ and\/ } [ 1, \Tk ] \text{ are collinear if and only if\/ } t(1,\Tg,\Th )= \Tk, \text{ i.e., } \Tg+\Th=\Tk,
$$
and their common line is $\langle \Tg,\Th\rangle$; and 
$$
[ \Tg],\ [\Th,0 ], \text{ and\/ } [ 0, \Tk] \text{ are collinear if and only if\/ } t(\Th,\Tg,\Tk )= 0, \text{ i.e., } \Th\diamond\Tg=\Tk,
$$
with common line $\langle \Tg, \Tk \rangle$.
\end{lem}

\begin{proof} 
By Equation \eqref{E:coordinates} the point $[m] \in \langle m, w \rangle$ and, since $t(0,m,w)=w$, also $[0, w] \in \langle m, w \rangle.$  Finally, $[x, y] \in \langle m, w \rangle$ if and only if $t(x,m,w)= y$.

In the particular cases $m = \Tg$.  A point $[0,y]$ in the line implies that $w=y$, which gives the values of $w$ in each case.  The algebraic criterion for collinearity is inferred from the coordinates of the third point.
\end{proof}

\begin{lem}\label{intersectionofthreelines} 
Let $\fT$ be a ternary ring and let $\Tg$, $\Th$ and  $\Tk$ be nonzero elements of $\fT$.  
In $\PP_\fT$ the lines $\langle \Tg \rangle$, $\langle \Th, 0 \rangle$, and $\langle 0, \Tk \rangle$ intersect in a point if and only if $t(\Tg,\Th,0)=\Tk$, i.e., $\Tg \times \Th = \Tk$.  The point of intersection is $[\Tg, \Tk]$.
\end{lem}

\begin{proof} 
The lines $\langle \Th,0 \rangle$ and $\langle 0, \Tk \rangle$ are concurrent in the point $[x,y]$ such that $t(x,\Th,0)=y$ and $t(x,0,\Tk)=y$.  That is, $y=x\times\Th$ and $y=\Tk$, so the point is $[x,\Tk]$ where $x\times\Th=\Tk$.  This equation has a unique solution since $\Th,\Tk\neq0$ and $\fT^\times$ is a quasigroup.  It follows that $[x,\Tk]\in\langle\Tg\rangle \iff x=\Tg \iff \Tg\times\Th=\Tk$.  That establishes the condition for concurrence and the location of the concurrence. 
\end{proof}

The \emph{dual plane} of $\PP$ is the plane $\PP^*$ whose points and lines are, respectively, the lines and points of $\PP$.  The dual coordinate system to $u,v,o,e$ in $\PP$ is $o^*,u^*,v^*,e^*$ in $\PP^*$ given by
\begin{gather*}
o^*=ou, 	\qquad 	u^*=ov, 	\qquad 	v^*=uv, 	\qquad 	e^*=\langle1,1\rangle	
\intertext{($e^*$ is the line $[0,1][1]$ spanned by $[0,1]$ and $[1]$) as well as $i^*=oe$ and }
(ou)^*=o, 	\quad 	(ov)^*=u, 	\quad 	(uv)^*=v, 	\quad	(oe)^*=i.
\\\label{E:dualcoords}
\end{gather*}

With this dual coordinate system, $[x,y]^*=\langle x,y \rangle$, $[x]^*=\langle x \rangle$, $\hz^*=\hZ$, and $\PP_\fT^* = \PP_{\fT^*}$ (see Martin \cite{Martin}); that explains our definition of $t^*$.  

\subsubsection{Abstract $3$-nets}\label{absnets}\

An (abstract) \emph{$3$-net} $\cN$ is an incidence structure of points and lines that consists of three pairwise disjoint pencils of lines (a \emph{pencil} is a family of pairwise nonintersecting lines), $\cL_{12}$, $\cL_{23}$, and $\cL_{13}$, such that every point is incident with exactly one line in each pencil.  (See \cite[Section 8.1]{Latin}, \cite[p.\ 141]{Demb}, et al.  Note that our $3$-net is labelled since we have named the three pencils.  If we ignore the names it is unlabelled.)  A \emph{subnet} of $\cN$ is a $3$-net whose points and lines are points and lines of $\cN$.

There are correspondences among quasigroups, $3$-nets, and biased expansions of $K_3$.  Indeed the latter two are essentially the same.

In a $3$-net $\cN$ every pencil has the same number of lines.  Choose bijections $\beta_{ij}$ from the pencils to an arbitrary set $\fQ$ of the same cardinality and define $\Qg \cdot \Qh=\Qk$ in $\fQ$ to mean that $\beta_{12}\inv(\Qg), \beta_{23}\inv(\Qh), \beta_{13}\inv(\Qk)$ are concurrent.  This defines a quasigroup operation.  The same net gives many quasigroups due to the choice of bijections and labelling; the net with unlabelled parallel classes corresponds one-to-one not to $\fQ$ but to its isostrophe class, while the net with parallel classes labelled corresponds to the isotopy class of $\fQ$ (see \cite[Theorem 8.1.3]{Latin}).  We shall assume parallel classes are labelled.

Conversely, given a quasigroup $\fQ$ we can construct a unique $3$-net $\cN(\fQ)$ by taking three sets $\cL_{ij}:=\fQ\times\{ij\}$ for $ij=12,23,13$, whose elements we call lines and denote by $L_{ij}(\Qg)$ (meaning $(\Qg,ij)$), and defining lines $L_{12}(\Qg)$, $L_{23}(\Qh)$, and $L_{13}(\Qk)$ to have an intersection point, which we may call $(\Qg,\Qh,\Qk)$, if and only if $\Qg\cdot\Qh=\Qk$.  Clearly, if $\fQ'$ is constructed from $\cN(\fQ)$ using the same underlying set $\fQ$, $\fQ$ and $\fQ'$ are isotopic.

It is easy to prove that a $3$-subnet of $\cN(\fQ)$ corresponds to a subquasigroup of $\fQ$ (after suitable isotopy of the latter).

From $\cN$ we can also construct a biased expansion $\Omega(\cN)$ of $K_3$:  each line $l_{ij}\in \cL_{ij}$ is regarded as an edge with endpoints $v_i,v_j$, and a triangle $l_{12}l_{23}l_{13}$ is balanced if and only if the three lines are concurrent in a point.  Thus, $\Omega(\cN)$ has node set $\{v_1,v_2,v_3\}$ and edge set $\cN$, and the points of $\cN$ are in one-to-one correspondence with the balanced circles of $\Omega(\cN)$.  Reversing the process we get a $3$-net $\cN(\Omega)$ from a biased expansion $\Omega$ of $K_3$.  
We summarize the correspondences in a diagram:  
$$
\xymatrix{
\fQ \ar[r]\ar[d]	&\cN(\fQ) \ar@{<->}[d]	\\
\fQ K_3 \ar@{=}[r]\ar@{<->}[ur]	&\Omega(\cN(\fQ))
}
$$

\begin{prop} \label{PD:net}
Given a $3$-net $\cN$, $\Omega(\cN)$ is a biased expansion of $K_3$.  Given $\Omega$, a biased expansion of $K_3$, $\cN(\Omega)$ is a $3$-net.  Also, $\Omega(\cN(\Omega))=\Omega$ and $\cN(\Omega(\cN))=\cN$. 
\end{prop}

\begin{proof}  
The proposition follows directly from the various definitions.  The construction is reversible so it is a bijection.  
\end{proof}

Since $K_3$ is labelled by having numbered nodes, technically we have a labelled proposition.  The unlabelled analog is an immediate corollary.

\begin{prop}  \label{PD:subnet}
Let $\cN=\cN(\fQ)$ be a $3$-net with corresponding biased expansion graph $\Omega(\cN) = \bgr{\fQ K_3}$.  
The $3$-subnets of $\cN$ have the form $\cN(\fQ_1)$ where $\fQ_1$ is a subquasigroup of an isotope of $\fQ$ and they correspond to the biased expansions of $K_3$ that are subgraphs of $\Omega(\cN)$.
These subgraphs are the balance-closed subgraphs of $\Omega(\cN)$ that are spanning and connected and they are also the quasigroup expansions of the form $\bgr{\fQ_1 K_3}$.
\end{prop}

This is a more abstract analog of \cite[Corollary 2.4]{BG6}(I).  

\begin{proof}  
We show that the subgraphs described in the lemma are precisely the subgraphs that are biased expansions of $K_3$.  Then the proposition follows from the correspondence between biased expansions and $3$-nets.

Let $\Omega_1 \subseteq \Omega(\cN)$ be a biased expansion of $K_3$.  $\Omega_1$ is spanning and connected.  There are no balanced digons because every balanced circle in $\Omega(\cN)$ is a triangle.  $\Omega_1$ is balance-closed because, since it is a biased expansion, it must contain the third edge of any balanced triangle of which it contains two edges.

Conversely, assume the subgraph $\Omega_1$ is connected, spanning, and balance-closed.  Being connected, it contains two nonparallel edges $e, f$.  Being balance-closed, it also contains the third edge of the unique balanced triangle containing $e$ and $f$; thus, no edge fiber of $\Omega_1$ is empty.  Since it cannot contain a balanced digon, it is consequently a biased expansion.
\end{proof}

The three correspondences among $3$-nets, isostrophe classes of quasigroups, and biased expansions of $K_3$ are obviously compatible with each other.

\subsubsection{Projective $3$-nets}\label{projnets}\

A $3$-net $\cN$ is \emph{embedded} in a projective plane $\PP$ if it consists of lines that each contain exactly one of a fixed set of three points (which we call the \emph{centers}) and all the intersection points of the lines other than the centers.  Every point of $\cN$ is on exactly three embedded lines of the net and no point of $\cN$ is on a line spanned by the centers.  (This kind of embedding is called \emph{regular}, a term we omit since we consider no other kind.  Irregular embeddings are interesting [Bogya, Korchm\'aros, and Nagy \cite{Nagy} is a representative paper] but they do not embed the matroid of $\cN$, which is $G(\fQ K_3)=L(\fQ K_3)$ where $\fQ$ is a quasigroup derived from $\cN$.)  
A $3$-net may be (regularly) embedded in $\PP$ in either of two ways.  An \emph{affine $3$-net} has collinear centers; equivalently, it consists of three parallel classes in an affine plane.  A \emph{triangular $3$-net} has noncollinear centers.  

A \emph{complete} $3$-net in $\PP$ consists of \emph{all} lines in $\PP$ that contain exactly one center and all points of $\PP$ not on a line spanned by the centers.  A $3$-net embedded in $\PP$ need not be complete, but obviously it is contained in a complete $3$-net.  

\subsubsection{Dual $3$-nets}\label{dualnets}\

A \emph{dual $3$-net} in $\PP$, $\cN^*$, consists of three lines, which we call \emph{main lines}, and some subset of the points in the union of the main lines excluding intersection points of the main lines, such that any two points of $\cN^*$ on different main lines generate a line (which we call a \emph{cross-line}) that meets the third main line in a point of $\cN^*$.  A cross-line contains exactly three points of $\cN^*$, one from each main line.  
A subset $A$ of $\cN^*$ is \emph{cross-closed} if the flat $A$ generates is empty, a point, or a cross-line $l$ such that $A = l \cap \cN^*$.
The main lines, points, and cross-lines are dual to the centers, lines, and points of $\cN$.  
The property in a quasigroup $\fQ$ derived from $\cN$ that corresponds to cross-lines is that $\Qg\cdot\Qh=\Qk$ if and only if the points mapped bijectively to $\Qg$, $\Qh$, and $\Qk$ are collinear.  
The dual $3$-net corresponds to $\Omega(\cN)$: the points, main lines, and cross-lines of $\cN^*$ correspond to the edges, edge fibers $p\inv(e_{ij})$, and balanced circles of $\Omega(\cN)$.  

The dual net is affine if the three main lines are concurrent, triangular if they are not.  

A dual $3$-net $\cN^*$ embedded in $\PP$ can be extended to a complete dual $3$-net of $\PP$ by taking all the points of $\PP$, except the centers, that are on the main lines of $\cN^*$.

\subsubsection{Partial $3$-nets}\label{p3nets}

The $3$-net interpretation of Theorem \ref{T:qx} seems interesting in its own right.  
First we have to define a \emph{partial $3$-net}.  It is a triple $\cN=(\cP,\cM,\cL)$ consisting of a point set $\cP$, a set $\cM=\{l_{12},l_{23},\l_{31}\}$ of three lines called \emph{main lines}, another set $\cL$ of \emph{short lines}, and an incidence relation ``in'' between points and line such that every point is in exactly one main line, every short line has exactly one point in common with each main line, and no two short lines have more than one point in common.  We wrote this as if a line were a set of points, but although a short line is determined by its set of points, a main line might not have any points.  A $3$-net is a partial $3$-net in which every main line contains a point and any pair of points in different main lines belongs to a short line.

Suppose another partial $3$-net is $\cN'=(\cP',\cM',\cL')$; we say $\cN$ is a partial sub-$3$-net of $\cN'$ if $\cP \subseteq \cP'$, $\cM=\cM'$, $\cL\subseteq \cL'$, and the incidence relation of $\cN$ is the restriction to $\cN$ of that of $\cN'$.  Theorem \ref{T:qx} restated in terms of nets is

\begin{cor}\label{C:p3net}
Every finite partial $3$-net is a partial sub-$3$-net of a finite $3$-net.
\end{cor}

\subsubsection{Quasigroups vs.\ nets vs.\ biased expansions vs.\ matroids}\label{qnbxm}\

Our treatment of plane embeddings abounds in cryptomorphisms (hidden structural equivalences).  For a quasigroup $\fQ$ we have a $3$-net $\cN(\fQ)$, a quasigroup expansion $\fQ K_3$, a full quasigroup expansion $\fQ K_3\full$, a biased graph $\bgr{\fQ K_3}$, a full biased graph $\bgr{\fQ K_3\full}$, and four matroids: $G(\fQ K_3)$, $\G(\fQ K_3)$, $L(\fQ K_3)$, and $L_0(\fQ K_3)$.  They are not all identical but they are all equivalent.  For instance, $\G(\fQ K_3)$ is identical to $G(\fQ K_3\full)$; biased-graph isomorphisms always imply matroid isomorphisms and, conversely, certain matroid isomorphisms imply biased-graph isomorphisms as we show in Section \ref{monomorph}; isomorphisms of full frame matroids $\G$ imply isomorphisms of the non-full frame matroids $G$; etc.  It is particularly important to be clear about the equivalence of projective-planar embeddings of the different structures, since the language we use to describe what are essentially the same embedding can vary depending on the circumstances.

For instance, an embedding of $\cN(\fQ)$ as a triangular $3$-net is the same thing as a representation of $\G(\fQ K_3)$ by lines, with the half edges mapping to the lines connecting the centers.  
Dually, an embedding of $\cN(\fQ)$ as a dual triangular $3$-net is the same thing as a point representation of $\G(\fQ K_3)$; the half edges map to the intersection points of the main lines.  
The sets $\hE(\hN)$ and $\hE\full(\hN)$ of points on the lines connecting the members of a basis $\hN$ are the points representing $\cN(\fT^\times)$ and equivalently $G(\fT^\times K_3)$, and $\cN\full(\fT^\times)$ and equivalently $\G(\fT^\times K_3)$.

An embedding of $\cN(\fQ)$ as an affine $3$-net is the same thing as a representation of $L_0(\fQ K_3)$ by lines; the extra point $e_0$ maps to the line that contains the centers.
An embedding of $\cN(\fQ)$ as a dual affine $3$-net is the same thing as a point representation of $L_0(\fQ K_3)$; the extra point $e_0$ maps to the point of concurrence of the main lines.

\subsubsection{Monomorphisms}\label{monomorph}\

We note here an enlightening aspect of the relationship between biased expansions of $K_3$ and their matroids.  
A biased graph determines its frame and lift matroids, by definition.  For quasigroup expansion graphs the converse is essentially true, by the surjective case of the two following results.  
A \emph{monomorphism} of biased graphs is a monomorphism of underlying graphs that preserves balance and imbalance of circles.
A \emph{long line} of a matroid is a line that contains at least four atoms.  

\begin{prop}\label{L:frame-graph}
If $\#\fQ > 1$, then any matroid monomorphism $\theta: \G(\fQ K_3) \embeds \G(\fQ' K_3)$ induces a unique biased-graph monomorphism $\theta': \bgr{\fQ K_3} \embeds \bgr{\fQ' K_3}$ that agrees with $\theta$ on $E(\fQ K_3)$.  When $\#\fQ = 1$, there is a biased-graph monomorphism $\bgr{\fQ K_3} \embeds \bgr{\fQ' K_3}$, but it may not be induced by $\theta$.
\end{prop}

\begin{proof}
First we assume $\#\fQ \geq 2$.  A balanced line in $\G$ has at most three points; therefore, we can identify any long line as the subgraph of all edges with endpoints in a pair of nodes.  Such lines exist when $\#\fQ>1$.  The points common to two long lines are the half edges; from that the incidences of all edges can be inferred.  As each of $ \G(\fQ K_3)$ and $\G(\fQ' K_3)$ has three long lines, those of the former must be mapped to those of the latter.  We conclude that $\theta|_{E(\fQ K_3)}$ induces an ordinary graph monomorphism $\theta'$.  

The three-point lines of $\G(\fQ K_3)$ tell us what balanced triangles there are.  They must be mapped to three-point lines of $\G(\fQ' K_3)$, which are also balanced triangles, so balance of triangles is preserved.  Imbalance of triangles is preserved because $\theta$ preserves rank, so an unbalanced triangle (rank $3$) cannot be carried to a balanced triangle (rank $2$).  All digons are unbalanced because there are no 2-circuits in the matroids.  Thus, $\theta'$ is a monomorphism of biased graphs.

If $\#\fQ = 1$, $\G(\fQ K_3) \cong G(K_4)$ so it is impossible to identify the half edges and nodes of $\fQ K_3$ from the full frame matroid.  However, in any embedding $G(K_4) \embeds \G(\fQ' K_3)$, there is some triangle that carries over to a balanced triangle and the remaining edges map to half edges of $\fQ' K_3$.
\end{proof}

\begin{prop}\label{L:lift-graph}
If $\#\fQ > 2$, then a matroid monomorphism $\theta: L_0(\fQ K_3) \embeds L_0(\fQ' K_3)$ induces a biased-graph monomorphism $\theta'': \bgr{\fQ K_3} \embeds \bgr{\fQ' K_3}$.  When $\#\fQ \leq 2$, there is a biased-graph monomorphism $\bgr{\fQ K_3} \embeds \bgr{\fQ' K_3}$, though it is not necessarily induced by $\theta$.
\end{prop}

\begin{proof}
The argument here is similar to that of Proposition \ref{L:frame-graph}.  
Assuming $\#\fQ \geq 3$, the long lines are precisely the lines that contain $e_0$, that is, the unbalanced lines of $L(\fQ K_3))$.  From this we identify $e_0$ as the unique element of all long lines, and we identify the parallel classes of edges as the remainders of the long lines.  The same holds in $L_0(\fQ' K_3)$, so $\theta''$ has to map $e_0$ to $e_0$ and edges to edges.  As it preserves parallelism, it is a graph monomorphism.  

The balanced triangles are the three-point lines of $L_0(\fQ K_3)$, so balance is preserved.  Imbalance of circles is preserved for the same reason as with the full frame matroid.

When $\#\fQ = 2$, $\fQ$ is isotopic to the sign group.  $L_0(\pm K_3) \cong F_7$, the Fano plane \cite{SBM}, so $e_0$ cannot be singled out from the points of the matroid, but the method of Proposition \ref{L:frame-graph} provides a biased-graph monomorphism.

When $\#\fQ = 1$, $e_0$ is the unique matroid coloop in $L_0(\fQ K_3)$.  The edges of $\fQ K_3$ must map into a line of at least three points, but this line need not be a balanced triangle in $\fQ' K_3$.  Still, those edges can be embedded in $\bgr{\fQ' K_3}$ by a different biased-graph monomorphism.
\end{proof}


\section{The planar full frame matroid and triangular $3$-nets}\label{q frame}

In this section we find somewhat algebraic necessary and sufficient conditions for the full frame matroid $\G(\fQ K_3)$ of a quasigroup expansion to be embeddable in a projective plane.  (For the non-full matroid $G(\fQ K_3)$ the conditions are broader because $G(\Omega)=L(\Omega)$; we postpone that matroid to Section \ref{q planar}.)  
We remind the reader that our results can be applied to any finite biased graph of order 3 because it can be embedded in a finite biased expansion of $K_3$ (Theorem \ref{T:qx}).  

An intuitive idea of the embedding is that, if $\fQ$ is a multiplicative subloop of a ternary ring, this ternary ring provides coordinates for $\G(\fQ K_3)$ in the corresponding projective plane.  (The explanation of the ``somewhat'' attached to ``algebraic'' is that the structure of a ternary ring is more geometrical than algebraic.)  Conversely, if $\G(\fQ K_3)$ embeds in a projective plane, then from four well-chosen points in the projective representation of $\G(\fQ K_3)$ we construct a ternary ring that contains $\fQ$.  The embeddings are Menel{\ae}an (if as points) and Cevian (if as lines) as in \mencev; the new contribution is a construction method, along with the ``somewhat algebraic'' criterion for their existence.

When $\fQ$ is a multiplicative subgroup of a skew field $\fF$, \cite[Theorem IV.2.1 and Corollary IV.2.4]{BG} guarantee that $\G(\fQ K_3)$ embeds as points and hyperplanes in the Desarguesian projective plane coordinatized by $\fF$, and conversely, embedding implies $\fQ$ is essentially such a subgroup.  Theorem \ref{parte1} presents similar results for all quasigroups and planes.  
In the latter part of this section we show that the new embeddings are the same as those obtained from skew fields, for quasigroups that are multiplicative subgroups of skew fields.

A remarkable conclusion is that, although an embedding of $\G(\fQ K_3)$ in $\PP$ by lines has a natural construction in terms of a ternary ring $\fT$ associated with $\PP$ (Section \ref{q f line}), there seems to be no general construction in terms of $\fT$ of a point embedding, analogous to that for Desarguesian planes.  
We had hoped to show a natural representation of $\G(\fQ K_3)$ by points in $\PP$ but we were unable to find one.  There is a kind of point embedding in $\PP$ but it is merely the dual of line embedding in $\PP^*$ and it is naturally expressed in terms of the dual ternary ring $\fT^*$ that coordinatizes $\PP^*$ (see Section \ref{q f point}).  
We can be more specific.  The conditions for $\G(\fQ K_3)$ to embed as lines seem natural because they involve the natural operations of addition and multiplication in $\fT$.  
By contrast, in a vector representation of $\G(\fQ K_3)$ \cite[Section IV.2]{BG} (where $\fQ$ is a subgroup of the multiplicative group of the skew field) one is forced to introduce negation, which is not defined in a ternary ring.  
We believe this supports the long-held opinion of Zaslavsky that frame representations by hyperplanes are more natural than those by points.  
That could not be demonstrated for biased graphs obtained from gains in $\fF^\times$ (in \cite[Part IV]{BG}) because linear duality is an isomorphism; it can be seen in planes because $\PP^*$ may not be isomorphic to $\PP$.  A point representation of a full frame matroid in $\PP$ in terms of $\fT$ would overthrow this conclusion, but we could not find any.\footnote{Possibly, defining $\fT^+$ in the Hall--Dembowski manner would reverse the conclusion.  Cf.\ Section \ref{ternary}.}

We denote the half edges in $\fT^\times K_3\full$ (or $\bar\fQ K_3\full$ if $\bar\fQ \subseteq \fT^\times$) and $\fQ K_3\full$ by $\Td_i$ and $\Qd_i$ in order to distinguish edges in the two different full frame matroids.

\subsection{Embedding as lines or points}\label{embed frame}\

We begin with criteria for embedding $3$-nets; equivalently, biased expansions of the triangle.  
Recall that $E_{ij}$ is the set of all edges joining $v_i$ and $v_j$.

\begin{lem} \label{L:netintoP}
{\rm(I)}  If $\fT$ is a ternary ring and $\bar\fQ$ is a subquasigroup of $\fT^\times$, then $\cN(\bar\fQ)$ embeds in $\PP_\fT$ as a triangular $3$-net and $G(\bar\fQ K_3)$ embeds in $\PP_\fT$ as lines.  Embedding functions $\bar\rho^*$ and $\rho^*$ are defined by
\begin{align*}
\bar\rho^*(L_{12}(\Tg)) &= \rho^*(\Tg e_{12}) = \langle \Tg \rangle, \\ 
\bar\rho^*(L_{23}(\Th)) &= \rho^*(\Th e_{23}) = \langle \Th, 0 \rangle, \\ 
\bar\rho^*(L_{13}(\Tk)) &= \rho^*(\Tk e_{13}) = \langle 0, \Tk \rangle, 
\end{align*}
for $\Tg, \Th, \Tk \in \bar\fQ$.  
The centers are $o=[0,0]$ for $\bar\rho^*(\cL_{12}) = \rho^*(E_{12})$, $u=[0]$ for $\bar\rho^*(\cL_{23}) = \rho^*(E_{23})$, and $v=\hz$ for $\bar\rho^*(\cL_{13}) = \rho^*(E_{13})$.  

The matroid embedding extends to $\G(\bar\fQ K_3)$ by 
$$
\rho^*( \Td_1 ) = \langle 0 \rangle , \quad \rho^* ( \Td_2 ) = \langle 0, 0 \rangle , \quad \rho^*( \Td_3 ) = \hZ.
$$

{\rm(II)}  By changing $\langle\ldots \rangle$ to $[\ldots ]^*$ we get $\cN(\bar\fQ)$ embedded in $\PP_\fT^*$ as the points of a dual $3$-net and the extended embedding becomes a point representation of $\G(\bar\fQ K_3)$ in $\PP_\fT^*$.  
\end{lem}

\begin{proof}
The work was done in Lemma \ref{intersectionofthreelines}.  The lines $\rho^*(\Tg e_{12})$, $\rho^*(\Th e_{23})$, $\rho^*(\Tk e_{13})$ are concurrent $\iff$ (by that lemma) $\Tg \times \Th = \Tk$ in $\fT^\times$ $\iff$ $\Tg \cdot \Th = \Tk$ in $\bar\fQ$.
\end{proof}

The multiplicity of notations for the same things can be confusing, but we cannot avoid it since each reflects a different point of view that plays a role in our theory.  
In particular, $o^*=ou=\langle0,0\rangle$, $u^*=ov=\langle0\rangle$, and $\hz^*=uv=\hZ$.  
The first notation is that of points in $\PP^*$.  The second is that of lines in $\PP$.  The third is affine coordinates in the coordinate system determined by $u,v,o,e$ and the ternary ring derived therefrom.  

We connect quasigroup expansions to geometry in another way.  
Let $\PP$ be a projective plane coordinatized by a planar ternary ring $\fT(u,v,o,e)$ associated to points $u,v,o,e$.  
Let $\cN^\times$ denote the complete triangular $3$-net in $\PP$ on centers $u,v,o$ and let $\cN^\times{}\full := \cN^\times \cup \{ou, ov, uv\}$; that is, we adjoin to $\cN^\times$ the lines joining the centers.  $\cN^\times$ is a projective realization of the abstract $3$-net $\cN(\fT^\times)$.  Thus they have essentially the same biased graph, written $\Omega(\cN^\times)$ or $\Omega(\cN(\fT^\times))$ depending on the version of $3$-net from which we obtained the biased graph.  Both are naturally isomorphic to $\bgr{\fT^\times K_3}$.  

Now take $\hN:=\{o,u,v\}$, a (matroid) basis for the plane.  Its dual is $\hN^*=\{uo, vo, uv\}$, which is both a set of lines in $\PP$ and a set of points in the dual plane $\PP^*$.  
Let $\hE\full(\hN^*)$ be the set of all lines of $\PP$ on the points of $\hN$; it is identical to the augmented projective $3$-net $\cN^\times{}\full$ and its biased graph $\Omega(\hN^*,\hE\full(\hN^*))$ is therefore the same as $\Omega(\cN(\fT^\times{}\full))$.  
Simultaneously, $\hE\full(\hN^*)$ is a set of points in $\PP^*$; as such it has a matroid structure.  
(All this is the viewpoint of \cite[Section 2]{BG6}.)    
For the special case of Lemma \ref{L:netintoP} in which $\bar\fQ=\fT^\times$ there is the simplified statement in Corollary \ref{CD:menplane}.  
After these explanations, it is essentially a statement of equivalence of notations, codified by $\rho^*$ and $\bar\rho^*$.

\begin{cor} \label{CD:menplane}
Let $\PP$ be a projective plane coordinatized by a ternary ring $\fT(u,v,o,e)$ associated to points $u,v,o,e$.  Let $u,v,o$ generate the complete triangular $3$-net $\cN^\times$.  
Then we have isomorphisms 
$$
\xymatrix{
\bgr{\fT^\times K_3\full} \ar[r]^(.4){\rho^*} 
&\Omega(\hN^*,\hE\full(\hN^*))
&\Omega\full(\cN(\fT^\times)) \ar[l]_(.53){\bar\rho^*} \cong \Omega(\cN^\times).
}
$$
and 
$$
\G\big(\Omega(\hN^*,\hE\full(\hN^*))\big) \cong M(\hE\full(\hN^*))
$$
through the natural correspondence of edges to projective points, $e \mapsto \he \in \PP^*$.
\end{cor}

The next result provides criteria for (regular) embedding of a $3$-net in a plane as a triangular net.

\begin{prop} \label{L:trinet}
Let $\fQ$ be a quasigroup, $\cN(\fQ)$ its $3$-net, and $\PP$ a projective plane.
\begin{enumerate}[{\rm(I)}]

\item The following properties are equivalent:
\begin{enumerate}[{\rm(a)}]
\item $\cN(\fQ)$ embeds as a triangular $3$-net in $\PP$.
\item There is a ternary ring $\fT$ coordinatizing $\PP$ such that $\fQ$ is isostrophic to a subquasigroup of $\fT^\times$.
\item There is a ternary ring $\fT$ coordinatizing $\PP$ such that $\fQ$ is isotopic to a subloop of $\fT^\times$.
\item (If $\fQ$ is a loop.)  There is a ternary ring $\fT$ coordinatizing $\PP$ such that $\fQ$ is isomorphic to a subloop of $\fT^\times$.
\end{enumerate}

\item The following properties are also equivalent:
\begin{enumerate}[{\rm(a$^*$)}]
\item $\cN(\fQ)$ embeds as a triangular dual $3$-net in $\PP$.
\item There is a ternary ring $\fT$ such that $\PP \cong \PP_\fT^*$ and $\fQ$ is isostrophic to a subquasigroup of $\fT^\diamond$.
\item There is a ternary ring $\fT$ such that $\PP \cong \PP_\fT^*$ and $\fQ$ is isotopic to a subloop of $(\fT^\diamond)\opp$.
\item (If $\fQ$ is a loop.)  There is a ternary ring $\fT$ such that $\PP \cong \PP_\fT^*$ and $\fQ$ is isomorphic to a subloop of $(\fT^\diamond)\opp$.
\end{enumerate}
\end{enumerate}
\end{prop}

Recall from Lemma \ref{P:Tlinear} that when $\fT$ is linear, the condition that $\fQ$ is isostrophic to a subquasigroup of $\fT^\diamond$, or isotopic or isomorphic to a subloop of $(\fT^\diamond)\opp$, can be replaced by the condition that $\fQ$ is, respectively, isostrophic to a subquasigroup of $\fT^\times$, or isotopic or isomorphic to a subloop of $(\fT^\times)\opp$.

(A technical problem:  We would like to be able to say in (c) and (c${}^*$) that $\fQ$ is isomorphic to a subquasigroup, but we do not know that to be true.  We leave it to people more expert in ternary rings to decide.)

\begin{proof}
The corresponding statements in (I) and (II) are equivalent by duality, the fact that $(\fT^\diamond)\opp = (\fT^*)^\times$ by Equation \eqref{E:diamonddual}, and the fact that a quasigroup is isostrophic to its own opposite.

Clearly, (d) $\implies$ (c) $\implies$ (b).  

Suppose (b) $\fQ$ is isostrophic to $\bar\fQ \subseteq \fT^\times$.  Since $\cN(\fQ)$ is invariant under isostrophe, we may as well assume $\fQ=\bar\fQ$.  Then we apply Lemma \ref{L:netintoP} to deduce (a).  

Now, assuming (a) $\cN(\fQ)$ embeds as a triangular $3$-net $\cN$ in $\PP$, we prove (d) if $\fQ$ is a loop and (c) in any case.  If $\fQ$ is not a loop, convert it to a loop by an isotopism.  We may identify $\cN(\fQ)$ with the embedded net $\cN$ since they are isomorphic.  
Write $L_{ij}(\Qg)$ for the line in $\cN_{ij}$ that corresponds to $\Qg \in \fQ$ and let the centers be $o,u,v$, defined as the common points of the pencils $\cN_{12}$, $\cN_{23}$, $\cN_{13}$, respectively.  
Let $e := L_{12}(1_\fQ) \wedge L_{23}(1_\fQ)$.  Then $i = L_{13}(1_\fQ) \wedge ou$.  Also, $e \in L_{13}(1_\fQ)$ because the three lines $L_{ij}(1_\fQ)$ of the identity in $\fQ$ are concurrent.  
This choice defines a coordinate system using the ternary ring $\fT(u,v,o,e)$ in which the lines of $\cN_{12}$ meet $ou$ in points other than $o$ and $u$; also, $i=[1,0]$ so that $1_\fQ$ is identified with $1\in\fT$.  

The exact correspondence is that the lines $L_{12}(\Qg)$, $L_{23}(\Qh)$, $L_{13}(\Qk)$ in $\cN(\fQ)$ become $\langle \Tg \rangle$, $\langle \Th, 0 \rangle$, $\langle 0, \Tk \rangle$ in the plane.  
To prove that, first consider $L_{12}(\Qg)$.  It contains $v$ and meets $ou$ in a point other than $o$ and $u$, so it corresponds to a projective line $\langle a \rangle$ for some $a\in\fT^\times$.  Define $\Tg_{12}:=a$.  
Next, consider $L_{23}(\Qh)$.  It contains $o$ and meets $uv$ in a point other than $u$ and $v$; thus, it corresponds to a projective line of the form $\langle m,0 \rangle$.  Define $\Th_{23}:=m$.
Finally, consider $L_{13}(\Qk)$.  It contains $u$ and meets $ov$ in a point other than $o$ and $v$; so it corresponds to a projective line $\langle 0,b \rangle$.  Define $\Tk_{13}:=b$.

We now have three mappings $\fQ \embeds \fT^\times$.  ($0\in\fT$ is not in the range of any of them because none of the points $o,u,v$ is a point of $\cN$.)  We should prove they are the same, but first we prove the operations in $\fQ$ and in $\fT^\times$ agree.  
Consider $\Qg,\Qh,\Qk\in\fQ$.  They satisfy $\Qg\cdot\Qh=\Qk$ 
$\iff$ (by Lemma \ref{intersectionofthreelines}) the lines $L_{12}(\Qg), L_{23}(\Qh), L_{13}(\Qk)$ in $\cN(\fQ)$ are concurrent in a point of $\cN(\fQ)$ 
$\iff$ (because of how $\cN(\fQ)$ is embedded in $\PP$) the lines $\langle \Tg_{12} \rangle, \langle \Th_{23}, 0 \rangle, \langle 0, \Tk_{13} \rangle$ in $\PP$ are concurrent in a point of $\PP$ 
$\iff$ (by Lemma \ref{intersectionofthreelines}) $\Tg_{12}\times\Th_{23} = \Tk_{13}$.  
The identity of $\fQ$ is carried to that of $\fT^\times$ by all three mappings because the point $e$ is the point of concurrence of the identity lines in all three pencils.  It follows that $1 \times \Th_{23} = \Tk_{13}$, equivalently $1_\fQ \cdot \Qh = \Qk$, i.e., $\Qh=\Qk$; therefore, $\Th_{23}=\Th_{13}$ for $\Qh \in \fQ$.  Similarly, $\Tg_{12}=\Tg_{13}$ for $\Qg \in\fQ$.  This proves the three mappings are the same and hence that $\Qg\cdot\Qh=\Qk$ $\iff$ $\Tg\times\Th = \Tk$.  
Thus, $\fQ$ is (embedded as) a subloop of $\fT^\times$; allowing for the initial isotopism, if any, we obtain (c) and (d).
\end{proof}

Our main purpose is to characterize matroid embedding.  That is the next result.  It is mainly a reinterpretation of Proposition \ref{L:trinet}.

\begin{thm} \label{parte1}
Let $\PP$ be a projective plane and let $ \fQ$  be a quasigroup.  
\begin{enumerate}[{\rm (I)}]

\item The following properties are equivalent:
\begin{enumerate}[{\rm(a)}]
\item $\G(\fQ K_3)$ embeds as lines in $\PP$.
\item $G(\fQ K_3)$ embeds as a triangular $3$-net of lines in $\PP$.
\item $\G(\fQ K_3)$ embeds as points in $\PP^*$.
\item $G(\fQ K_3)$ embeds as a dual triangular $3$-net of points in $\PP^*$.
\item There is a ternary ring $\fT$ coordinatizing $\PP$ such that $\fQ$ is isostrophic to a subquasigroup of $\fT^\times$.
\item There is a ternary ring $\fT$ coordinatizing $\PP$ such that $\fQ$ is isotopic to a subloop of $\fT^\times$.
\item (If $\fQ$ is a loop.)  There is a ternary ring $\fT$ coordinatizing $\PP$ such that $\fQ$ is isomorphic to a subloop of $\fT^\times$.
\end{enumerate}

\item The following properties are also equivalent:
\begin{enumerate}[{\rm(a)}]
\item $\G(\fQ K_3)$ embeds as points in $\PP$.
\item $G(\fQ K_3)$ embeds as the points of a dual triangular $3$-net in $\PP$.
\item There is a ternary ring $\fT$ such that $\PP \cong \PP_\fT$ and $\fQ$ is isostrophic to a subquasigroup of $\fT^\diamond$.
\item There is a ternary ring $\fT$ such that $\PP \cong \PP_\fT$ and $\fQ$ is isotopic to a subloop of $(\fT^\diamond)\opp$.
\item (If $\fQ$ is a loop.)  There is a ternary ring $\fT$ such that $\PP \cong \PP_\fT$ and $\fQ$ is isomorphic to a subloop of $(\fT^\diamond)\opp$.
\end{enumerate}
\end{enumerate}
\end{thm}

We remind the reader that $(\fT^*)^\times = (\fT^\diamond)\opp$.  It makes sense to view (IIc, d) as embedding $\fQ\opp$ in the points of $\PP$ since then we can simply say $\fQ\opp$ is isotopic or isomorphic to a subloop of $\fT^\diamond$.

\begin{proof}
First we prove that $\G(\fQ K_3)$ embeds as a matroid of lines if and only if $\cN(\fQ)$ embeds as a triangular $3$-net.  

Suppose $\cN(\fQ)$ embeds in the plane as a triangular $3$-net.  By adjoining the lines generated by the centers we obtain an embedding of $\G(\fQ K_3)$.  

Conversely, suppose that $\G(\fQ K_3)$ embeds as lines.  The half edges $\Qd_i$ embed as lines $L_i$, which are nonconcurrent because the half edges have rank 3; this implies that $\codim(L_1\wedge L_2\wedge L_3) = \rk(L_1\wedge L_2\wedge L_3) = 3$.  As the matroid line spanned by $\Qd_i$ and $\Qd_j$ contains all edges $ge_{ij}$ (for $g\in\fQ$), the projective lines $L_{ij}(g)$ representing those edges contain the intersection point $P_{ij} := L_i \wedge L_j$; it follows that all the lines representing links $ge_{ij}$ concur in the point $P_{ij}$, which is therefore the center for one pencil of a $3$-net $\cN$ that consists of all the projective lines $L_{ij}(g)$ (for $1\leq i < j \leq 3$ and $g\in\fQ$).  The net is triangular because, if the centers $P_{ij}$ were collinear, the lines $L_i$ would coincide and not have rank 3.  
Thus, $\G(\fQ K_3)$ embeds as the triangular $3$-net $\cN$ of projective lines together with the three lines $l_i$ generated by the three noncollinear centers $P_{ij}$, and in particular, that gives an embedding of $\cN(\fQ)$ as a triangular $3$-net in $\PP$.

That proves the desired equivalence; we obtain (I) by applying Proposition \ref{L:trinet}.  Part (II) is the dual of (I) with $\fT^*$ changed to $\fT$.
\end{proof}

We could possibly have shortened the proof of Theorem \ref{parte1} by going directly from the matroid to the plane and its ternary rings, but we believe our three-step proof sheds more light on the relationships among biased expansion graphs, $3$-nets, and planes and that the intermediate results of Proposition \ref{L:trinet} about $3$-nets are independently interesting.

\begin{prob}\label{Pr:subqgp}
Can we replace isotopy to a subloop in parts (I)(c) of Proposition \ref{L:trinet} and (I)(f) of Theorem \ref{parte1} by isomorphism to a subquasigroup?  Though it appears unlikely, we could not be sure.
\end{prob}

When $\fT$ is linear, in Theorem \ref{parte1} isostrophe of $\fQ$ to a subquasigroup of $\fT^\diamond$ can be replaced by isostrophe of $\fQ$ to a subquasigroup of $\fT^\times$ and isotopy or isomorphism of $\fQ$ to a subloop of $(\fT^\diamond)\opp$ can be replaced by isotopy or isomorphism of $\fQ$ to a subloop of $(\fT^\times)\opp$.

The reader should observe that, according to Theorem \ref{parte1}, there is no automatic connection between point and line embeddability of $\G(\fQ K_3)$ in $\PP$.

We conclude the general treatment by interpreting cross-closure of a subset of a point representation of $\G(\fQ K_3)$ in terms of the algebra of $\fQ$.  

\begin{cor} \label{CD:triplane}
Let $\hN$ be a basis for a projective plane $\PP$ and let $\fQ$ be a quasigroup.  
Suppose $\cN(\fQ)$ is embedded as a dual triangular $3$-net $\hE \subseteq \hE(\hN)$ in $\PP$; equivalently, $\hE$ is a point representation of  $\G(\fQ K_3)$.  
If $\hE_1$ is a cross-closed subset of $\hE$ that is not contained in a line, then there is a loop $\fQ_1$ that is a subloop of a principal loop isotope $\fQ'$ of $\fQ$ and is such that $\bgr{\fQ_1 K_3} \cong \Omega(\hN,\hE_1)$ and $\G(\fQ_1 K_3) \cong M(\hE_1 \cup \hN)$ by the natural correspondence $e \mapsto \he$.  
\end{cor}

We remind the reader that principally isotopic quasigroups generate the same $3$-net, so $\cN(\fQ') = \cN(\fQ)$, and also that if $\hN$ is a basis for $\PP$, then $\hE(\hN)$ consists of all the points on the lines joining the basis elements except the basis elements themselves; and that an embedding of $\cN(\fQ)$ as a dual triangular $3$-net is the same thing as a point representation of $\G(\fQ K_3)$ in which the half edges map to the members of $\hN$ (Section \ref{qnbxm}).

\begin{proof}
This is a reinterpretation of \cite[Corollary 2.4]{BG6}(I). 
It suffices to prove a loop $\fQ_1$ exists that makes $\bgr{\fQ' K_3} \cong \Omega(\hN,\hE_1)$; the matroid isomorphism follows automatically.

By \cite[Corollary 2.4]{BG6}(I), 
$\Omega(\hN,\hE_1)$ is a biased expansion of $K_3$, $K_2$, or $K_1$.  The first of these is the nontrivial case.  Proposition \ref{PD:subnet} implies that each such biased expansion has the form $\bgr{\fQ_1 K_3}$ where $\fQ_1$ is a subquasigroup of $\fQ$, and corresponds to a $3$-subnet $\cN(\fQ_1)$ of $\cN(\fQ)$.  By Lemma \ref{L:subisotope} we may isotope $\fQ$ and $\fQ_1$ to $\fQ'$ and $\fQ_1'$ such that the latter is a subloop of the former.  Then $\bgr{\fQ_1 K_3} \cong \Omega(\hN,\hE_1)$ under the same bijection $e \mapsto \he$ through which $\bgr{\fQ K_3} \cong \Omega(\hN,\hE)$.
\end{proof}

\subsection{Desarguesian planes}\label{q desargframe}\

In \cite[Section IV.2.1]{BG} we developed canonical representations of the frame matroid $G(\Omega)$ of a biased graph $\bgr{\Phi}$ derived from a gain graph $\Phi$ with gains in $\fF^\times$, where $\fF$ is a skew field, by vectors and also by hyperplanes in $\fF^n$.  (Canonical, for this discussion, means a representation that extends to the full frame matroid, whence we treat $\G(\Omega)$ in this section.)  We wish to compare those with the point and line representations of a quasigroup expansion $\Omega$ in this section in order to show both sections have the same representations with the same criteria for concurrence of lines and collinearity of points.  

In $\bbP^2(\fF)$ the ternary operation is $t_\fF(x,m,b) = xm+b$, computed in $\fF$; thus, $\fT^+ = \fF^+$ and $\fT^\times = \fF^\times$.  (For simplicity we omit the subscript on $t$ henceforth.)  
The dual operation is $t^*(x,m,b)=y \iff t(m,x,y)=b \iff mx+y=b \iff y=b-mx.$  
The diamond operation, defined by $t(x,y,x\diamond y)=0$, in terms of $\fF$ is $x\diamond y = -yx$.  In particular, $1\diamond x = x\diamond1 = -x$ in $\fF$; for this reason we regard $\diamond$ as expressing a ternary-ring analog of negation.

\subsubsection{One plane, two systems}\label{2systems}\

The first task is to show the equivalence of the two coordinate systems, skew-field and ternary-ring, of the plane $\bbP^2(\fF)$.  

We represent the group expansion $\fF^\times K_3$ in $\fF^3=\{(x_1,x_2,x_3): x_1,x_2,x_3 \in \fF\}$.  Projective coordinates in the plane $\bbP^2(\fF)$ are homogeneous coordinates $[x_1,x_2,x_3]$, the square brackets on a triple indicating that scaling the coordinates (multiplication on the left by a nonzero element of $\fF$) denotes the same projective point.  

The translation of homogeneous coordinates to the ternary ring coordinates we use for a plane $\PP_\fT$ is slightly complicated.  Vectors correspond to projective points by treating $x_1$ as the homogenizing variable so that, generically, $[x_1,x_2,x_3] = [1,x,y]$ where $[x,y]$ are affine coordinates.  The complete formulas are
\begin{align}\label{E:pointsFT}
\begin{aligned}{}
[x_1,x_2,x_3] &= [x_1\inv x_2,x_1\inv x_3] &\text{ if } x_1\neq0, \\
[0,x_2,x_3] &= [x_2\inv x_3] &\text{ if } x_2\neq0,  \\
[0,0,x_3] &= \hz &\text{ if } x_3\neq0.
\end{aligned}
\end{align}
where the second notation is that we use for a plane $\PP_\fT$.  
If $x_1=0$ the point we want is the ideal point $[m]$ on the lines of slope $m$ such that $x_3=x_2m+x_1b=x_2m$ in homogeneous coordinates; thus $m=x_2\inv x_3$.  If $x_1=x_2=0$ the point is simply the ideal point $\hz$ at slope $\infty$.  

Planes in $\fF^3$ are given by equations $x_1a_1+x_2a_2+x_3a_3=0$.  Homogeneous coordinates of planes, $\langle a_1,a_2,a_3 \rangle$, are obtained by scaling on the right, since that preserves the points of the plane.  These planes correspond to lines in the projective plane by 
\begin{align}\label{E:linesFT}
\begin{aligned}
\langle a_1,a_2,a_3 \rangle &= \langle -a_2a_3\inv,-a_1a_3\inv \rangle &\text{ if } a_3\neq0, \\
\langle a_1,a_2,0 \rangle &= \langle -a_1a_2\inv \rangle &\text{ if } a_2\neq0,  \\
\langle a_1,0,0 \rangle &= \hZ &\text{ if } a_1\neq0,
\end{aligned}
\end{align}
where the second notation is that for $\PP_\fT$: slope-intercept notation $\langle m,b \rangle$, $x$-intercept notation $\langle a \rangle$ (from $x=a$), or simply the ideal line $\hZ$.  
For the proof consider an ordinary point ($x_1=1$ after scaling) on the line.  
If $a_3\neq0$ the plane equation becomes $y = -x(a_2a_3\inv) - (a_1a_3\inv)$ so $m=-a_2a_3\inv$ and $b=-a_1a_3\inv$.  
If $a_3=0\neq a_2$ the equation becomes $x = -a_1a_2\inv$ so $a = -a_1a_2\inv$.  
Otherwise, we have the ideal line.

We omit the easy verification that the criteria for a point to be incident with a line in $\bbP^2(\fF)$ are the same whether computed in $\fF$-coordinates with addition and multiplication or in affine coordinates using $t$.  
It follows that the criteria for collinearity of three points or concurrence of three lines are the same in both methods of computation.

\subsubsection{Line representation}\label{q f line}\

We examine the hyperplanar representation first.  
We assume a quasigroup $\fQ$ that is isotopic to a subgroup $\bar\fQ$ of $\fT^\times = \fF^\times$.  Since $\fQ K_3 \cong \bar\fQ K_3$, we may as well assume $\fQ=\bar\fQ$.

The coordinatization of $\bbP^2(\fF)$ by $\fT$ gives a frame representation of $\fQ K_3$ by lines: 
\begin{align}\label{E:f lines}
\begin{aligned}
\langle \Qg \rangle \leftrightarrow \Qg e_{12},\ 
\langle \Qh,0 \rangle \leftrightarrow \Qh e_{23},\ 
\langle 0,\Qk \rangle \leftrightarrow \Qk e_{13}
\end{aligned}
\end{align}
in the system of Lemma \ref{L:netintoP}.  

On the other hand, $\fQ$ is a subgroup of $\fF^\times$, which therefore acts as a gain group for $\fQ K_3$.  
In \cite[Section IV.2.1]{BG} the edge $\Qg e_{ij}$, where $\Qg \in \fF^\times$, is represented by the $\fF^3$-plane whose equation is $x_j = x_i \Qg$.  
In Section \ref{2systems} we showed that planes in projective coordinates correspond to lines in affine coordinates by
\begin{align}\label{E:f lines D}
\begin{aligned}
\Qg e_{12} \leftrightarrow x_2 = x_1 \Qg &\leftrightarrow \langle \Qg,-1,0 \rangle = \langle \Qg \rangle, \\
\Qh e_{23} \leftrightarrow x_3 = x_2 \Qh &\leftrightarrow \langle 0,\Qh,-1 \rangle = \langle \Qh,0 \rangle, \\ 
\Qk\inv e_{31}=\Qk e_{13} \leftrightarrow x_3 = x_1 \Qk &\leftrightarrow \langle \Qk,0,-1 \rangle = \langle 0,\Qk \rangle. \\ 
\end{aligned}
\end{align}
This computation shows that an edge corresponds to the same lines in the system of \cite[Section IV.2.1]{BG} and that of this section.  
(Note that the line coordinate vectors $\langle a_1,a_2,a_3 \rangle$ in \eqref{E:f lines D} coincide with the right canonical representation of edges in \cite[Section IV.2.2]{BG}; that is because $\fF$ multiplies line coordinates on the right.)

The criterion for concurrency of the lines in \eqref{E:f lines} is that $\Qk = \Qg\times\Qh$ in $\fT$, which is $\Qk = \Qg\Qh$ in $\fF$.  
The skew field criterion for concurrence of the lines $x_2 = x_1 \Qg$, $x_3=x_2 \Qh$, $x_1=x_3 \Qk$ is that $\phi(\Qg e_{12})\phi((\Qh e_{23}))\phi((\Qk e_{13})\inv) = 1$, i.e., $\Qg\Qh=\Qk$ (\cite[Corollary IV.2.2]{BG}; again note the edge directions).  
In other words, the concurrency criterion from \cite[Part IV]{BG} is the same as that in the plane $\PP_\fT$.

\subsubsection{Point representation}\label{q f point}

Comparing the vector representation of \cite[Section IV.2.1]{BG} with the requirements of a point representation in a non-Desarguesian plane supports the belief that there is no natural representation of $\G(\fQ K_3)$ by points in $\PP$.
In \cite[Section IV.2.1]{BG} the vectors in $\fF^3$ for edges are
\begin{align}\label{E:f vectors}
\Qg e_{12} \leftrightarrow (-1,\Tg,0), \ 
\Qh e_{23} \leftrightarrow (0,-1,\Th), \ 
\Qk e_{13} \leftrightarrow (-1,0,\Tk)
\end{align}
and the vectors are dependent if and only if $\Tk = \Tg\Th$ in the gain group $\fF^\times$ (\cite[Theorem IV.2.1]{BG}; note the edge directions).    The question is how that compares with the diamond criterion of Proposition \ref{collinearpointsinaplane}.  The answer is not clear because it is not clear how $\cN(\fQ)$ should be embedded as points.

\medskip
\emph{Method 1.  Translate the Desarguesian representation.}
One method is simply to translate the representation of $\G(\fQ K_3)$ in \cite[Section IV.2.1]{BG} into affine coordinates using the system of Equations \eqref{E:pointsFT}.  The vectors \eqref{E:f vectors} become the projective points 
\begin{align}\label{E:f negpoints}
[-\Tg,0] \leftrightarrow \Qg e_{12},\ 
[-\Th] \leftrightarrow \Qh e_{23},\ 
[0,-\Tk] \leftrightarrow \Qk e_{13}.
\end{align}
By Proposition \ref{collinearpointsinaplane}, the points are collinear if and only if $(-\Tg)\diamond(-\Th)=(-\Tk)$.  Restated in $\fF$ this is $(-\Tg)(-\Th)+(-\Tk)=0$, i.e., $\Tg\Th=\Tk$, the same as the gain-group condition, as it should be.  

This translation suggests an embedding rule that assumes $\fQ$ is isotopic to a subquasigroup $\bar\fQ$ of $\fT^\times$ and maps edges to points in $\PP_\fT$ using the correspondence in Equations \eqref{E:f negpoints}.  That is not a satisfactory answer.  After all, what is $-\Tg$ in a ternary ring?  There are several possible definitions, starting with $x+(-x)=0$ and $(-x)+x=0$, which are inequivalent.  The one that should be used is the one that makes $(-\Tg)\diamond(-\Th)=(-\Tk)$ equivalent to the gain-group rule $\Tg\Th=\Tk$ (carried over from $\fQ$ by isotopy).  This equivalence cannot be expected in an arbitrary ternary ring; while it may be true in special kinds of ternary rings, we have not investigated the matter.  Thus, at best, the direct translation from skew fields is problematic.

\medskip
\emph{Method 2.  Dualize the line representation.}
The other obvious method is to dualize the line representation as in Lemma \ref{L:netintoP}(II), representing $\G(\fQ K_3)$ by points in $\PP_\fT^*$.  

The representing points are obtained by writing \eqref{E:f lines} with point notation instead of line notation:
\begin{align}\label{E:f Tpoints}
[ \Tg ]^* \leftrightarrow \Qg e_{12},\ 
[ \Th,0 ]^* \leftrightarrow \Qh e_{23},\ 
[ 0,\Tk ]^* \leftrightarrow \Qk e_{13} .
\end{align}
By Lemma \ref{collinearpointsinaplane} they are collinear if and only if $\Th \diamond^* \Tg = \Tk$; that is, $\Tg \times \Th = \Tk$, equivalent to $\Qg \cdot \Qh = \Qk$ in $\fQ$ by the isotopism.  
This agrees with the skew-field criterion for coplanarity of the vectors \eqref{E:f vectors} and it seems natural since it requires $\fQ$ be isotopic to a subquasigroup $\bar\fQ$ of $\fT^\times$.  
However, since it is merely the dual of the line representation, by adopting Method 2 we effectively give up the idea of a point representation in $\PP_\fT$.

\medskip
\emph{Method 3.  Change the projectivization.}
A third method might be to change the projectivization rule \eqref{E:pointsFT} and hope that suggests a general point representation in $\PP_\fT$.  We were unable to find a rule that works.  Variants of \eqref{E:pointsFT} led us to formulas with negation and, worse, reciprocation, neither of which exists in a usable way in all ternary rings.

\medskip
\emph{Conclusion.}  
There may be no natural system of point representation of $\G(\fQ K_3)$, governed by an embedding of $\fQ$ into a multiplicative structure in $\fT$ (whether $\times$ or $\diamond$ or some other), in a way that generalizes the embedding of a gain group $\fG$ into $\fF^\times$ as in \cite[Section IV.2]{BG}, that can apply to all ternary rings.  
At the very least, this remains an open problem.

\section{Planarity of the extended lift matroid and affine $3$-nets}\label{q lift}

Here we characterize the embeddability of the extended lift matroid of a quasigroup expansion $L_0(\fQ K_3)$ in a projective plane.  (See Section \ref{q planar} for the unextended matroid.)
An intuitive summary is that, if $\fQ$ is an additive subquasigroup of a ternary ring, this ternary ring provides the coordinates of $L_0(\fQ K_3)$ in the corresponding projective plane.  In the opposite direction, if $L_0(\fQ K_3)$ embeds in a plane, then four suitable points in its affine representation generate a ternary ring that contains $\fQ$ as an additive subquasigroup.

When $\fQ$ is an additive subgroup of a skew field $\fF$, \cite[Theorem IV.4.1]{BG} guarantees that $L_0(\fQ K_3)$ embeds in the Desarguesian projective plane coordinatized by $\fF$.  Theorem \ref{bmain} can be thought of as a planar generalization of that theorem.  

In $L_0(\fT^+ K_3)$ and $L_0(\fQ K_3)$ we denote the extra points by $\Tez$ and $\Qez$ in order to distinguish them from each other.

\subsection{Embedding as lines or points}\label{embed lift}\

We begin with criteria for embedding $3$-nets, equivalently biased expansions of $K_3$.  Again, remember that $G(\fQ K_3) = L(\fQ K_3)$ but $\G(\fQ K_3) \neq L_0(\fQ K_3)$.

\begin{lem} \label{embeddingofTKintoP}
If $\fT$ is a ternary ring and $\bar\fQ$ is a subquasigroup of $\fT^+$, then $\cN(\bar\fQ)$ embeds in $\PP_\fT$ as a dual affine $3$-net and $L(\bar\fQ K_3)$ embeds as points.  An embedding function $\theta$ is defined by
$$
\bar\theta(L_{12}(\Tg)) = \theta( \Tg e_{12}) = [\Tg], \quad 
\bar\theta(L_{23}(\Th)) = \theta(\Th e_{23}) = [0,\Th], \quad 
\bar\theta(L_{13}(\Tk)) = \theta(\Tk e_{13}) = [1,\Tk], \quad
$$
for $\Tg,\Th,\Tk \in \bar\fQ$.  The main lines of the embedding are $\bar\theta(\cL_{12}) = \hZ$, $\bar\theta(\cL_{23}) = \langle0\rangle$, and $\bar\theta(\cL_{13}) = \langle1\rangle$, concurrent in the point $\hz$.  

The embedding extends to $L_0(\bar\fQ K_3)$ with $\theta$ extended by $\theta(\Tez) = \hz$.  
\end{lem}

\begin{proof}
We give the proof for $\bar\theta$, which implies that for $\theta$ because $\cN(\bar\fQ)$ is essentially equivalent to $L(\bar\fQ K_3)$.  
That means $\Omega(\cN(\bar\fQ))$ is naturally isomorphic to $\bgr{\bar\fQ K_3}$ by the mapping $\bar\theta \circ \theta\inv: L_{ij}(\Tg) \mapsto \Tg e_{ij}$ on edges, extended to nodes in the obvious way (as the identity if both base graphs $K_3$ are the same).

From the definition of $\theta$, 
\begin{align*}
\theta(\{\Tg e_{12}: \Tg \in \bar\fQ \} \cup  \{\Tez\} )  &  =   \big\{ [\Tg], \hz : \Tg \in \bar\fQ \big\} = \hZ, \\
\theta(\{\Th e_{23}: \Th \in \bar\fQ \} \cup  \{\Tez\})  &  =   \big\{ [0,\Th], \hz : \Th \in \bar\fQ \big\} = \langle 0 \rangle, \\
\theta(\{\Tk e_{13}: \Tk \in \bar\fQ \} \cup  \{\Tez\} )  &  =   \big\{ [1, \Tk], \hz : \Tk \in \bar\fQ \big\} = \langle 1 \rangle.
\end{align*}
This proves that a line of $L_0(\bar\fQ K_3)$ that is not a balanced circle is collinear in $\PP_\fT$.

The triangle $\{ \Tg e_{12}, \Th e_{23}, \Tk e_{13} \}$ in $\bar\fQ K_3$ is balanced $\iff$ $\Tg + \Th = \Tk$ $\iff$ (by Lemma \ref{collinearpointsinaplane}) $\theta(\{ \Tg e_{12}, \Th e_{23}, \Tk e_{13} \})= \{[ \Tg], [ 0, \Th ], [ 1, \Tk ] \}$ is collinear. 

Since there are no other lines in $L_0(\bar\fQ K_3)$, this completes the proof that $\theta$ embeds $L_0(\bar\fQ K_3)$ into $\PP_\fT$.
\end{proof}

A dual affine $3$-net induces a biased graph geometrically.  Suppose $\cN$ is a $3$-net embedded in $\PP$ as the dual affine $3$-net $\cN^*$ with main lines $l_{12}, l_{23}, l_{13}$, concurrent at $p$.  Construct a graph with node set $N(K_3)$ and with an edge $e_{ij}$ for each point of $\cN^* \cap l_{ij}$; this is the underlying graph.  By definition, each cross-line that meets $\cN^*$ at more than one point meets it at three points, one in each main line, corresponding to a triple of concurrent lines in $\cN$.  This triple is a balanced circle in the biased graph.  We write $\Omega(\cN^*,p)$ for this biased graph, which is a biased expansion of $K_3$ because by construction it is isomorphic to $\Omega(\cN)$.

In the special case where $\cN = \cN(\fT^+)$, the projective realization $\cN(\fT^+)^*$ is a complete dual affine $3$-net that consists of all points on the main lines $ov, iv, uv$, other than $v$ itself; we call this dual $3$-net $\cN^+$.  In terms of Lemma \ref{embeddingofTKintoP}, $\bar\fQ = \fT^+$ and $\cN(\fT^+)$ embeds as $\cN^+$ with main lines $ov=\theta(E_{12})$, $uv=\theta(E_{23})$, $iv=\theta(E_{13})$; that is, $\cN(\fT^+)^* = \cN^+ = (ov \cup uv \cup iv) \setminus \{v\}$.  
Explicit isomorphisms are analogous to those of Corollary \ref{CD:menplane}.

\begin{cor}  \label{CD:orthoring}
Let $\PP$ be a projective plane coordinatized by a ternary ring $\fT=\fT(u,v,o,e)$ and let $i=ou \wedge ve$.  Let $o,u,i,v$ generate the complete dual affine $3$-net $\cN^+$ in $\PP$.   
Then 
we have isomorphisms
$$
\xymatrix{
\bgr{\fT^\times K_3} \ar[r]^(.45){\theta} 
&\Omega(\cN^+,v)
&\Omega(\cN(\fT^+)) \ar[l]_(.45){\bar\theta}
}
$$
and 
$L_0(\fT^+ K_3) \cong M(\cN^+ \cup \{v\})$ under $\theta$ extended by $\theta(\Qez) = \he_0 := v$. 
\end{cor}

Proposition \ref{L:affnet} gives an algebraic criterion for a $3$-net to be embeddable in $\PP$ as an affine $3$-net.

\begin{prop} \label{L:affnet} 
Let $\fQ$ be a quasigroup, $\cN(\fQ)$ its $3$-net, and $\PP$ a projective plane.
The following properties are equivalent:
\begin{enumerate}[{\rm(a)}]
\item The $3$-net $\cN(\fQ)$ embeds as a dual affine $3$-net in $\PP$.
\item There is a ternary ring $\fT$ coordinatizing $\PP$ such that $\fQ$ is isostrophic to a subquasigroup of $\fT^+$.
\item There is a ternary ring $\fT$ coordinatizing $\PP$ such that $\fQ$ is isotopic to a subloop of $\fT^+$.
\item (If $\fQ$ is a loop.)  There is a ternary ring $\fT$ coordinatizing $\PP$ such that $\fQ$ is isomorphic to a subloop of $\fT^+$.
\end{enumerate}
\end{prop}

\begin{proof}
The proof is similar to that of Proposition \ref{L:trinet} but with $\fT^+$ in place of $\fT^\times$.  

Clearly, (d) $\implies$ (c) (by isotoping $\fQ$ to a loop) $\implies$ (b).  

If $\fQ$ is isostrophic to a subquasigroup of $\fT^+$, we may as well assume it is a subquasigroup.   Then we apply Lemma \ref{embeddingofTKintoP} to infer (a).

Now we assume (a) $\cN(\fQ)$ embeds as a dual affine $3$-net in $\PP$.  We wish to prove (d) when $\fQ$ is a loop.  We identify $\cN(\fQ)$ with the embedded dual $3$-net $\cN^*$, whose main lines $L_{ij}$ correspond to the parallel classes $\cL_{ij}$ in $\cN(\fQ)$.  Since $\fQ$ is a loop with identity $1$, the images of the three lines $L_{ij}(1)$ of $\cN(\fQ)$ are collinear points of $\cN^*$ labeled $o_{12}\in L_{12}$, $o_{23}\in L_{23}$, and $o_{13}\in L_{13}$.  

Define $o:=o_{23}$, $u:=o_{12}$, and $i:=o_{13}$.  
Since $\cN^*$ is a dual affine $3$-net, there is a point $v$ common to all the main lines and $v$ is not a point of $\cN^*$.  Choose $e$ to be any point on $iv$ that belongs to $\cN^*$, other than $i$ and $v$.  We now have a coordinate system for $\Pi$ in terms of the ternary ring $\fT(u,v,o,e)$ in which $o_{12}=u=[0]$, $o_{23}=o=[0,0]$, and $o_{13}=i=[1,0]$.  The main line $L_{12}$ is $uv=\hZ$, while $L_{23}=ov=\langle0\rangle$ and $L_{13}=iv=\langle1\rangle$.  

In the embedding of $\cN(\fQ)$, the net lines $L_{12}(\Qg), L_{23}(\Qh), L_{13}(\Qk)$ corresponding to $\Qg, \Qh, \Qk \in \fQ$ are embedded as points with $\fT$-coordinates $[\Tg_{12}]$, $[0,\Th_{23}]$, $[1,\Tk_{13}]$.  This defines three mappings $\fQ \to \fT$ by $\Qg \mapsto \Tg_{12}$, $\Qh \mapsto \Th_{23}$, and $\Qk \mapsto \Tk_{13}$.  
In particular, $o_{12}=[0]$ implies that $1_{12}=0$, $o_{23}=[0,0]$ implies that $1_{23}=0$, and $o_{13}=[1,0]$ implies that $1_{13}=0$.

In $\cN(\fQ)$, $L_{12}(\Qg)$, $L_{23}(\Qh)$, $L_{13}(\Qk)$ are concurrent in a point $\iff \Qg \cdot \Qh = \Qk$, so in the plane embedding the corresponding points are collinear $\iff \Qg \cdot \Qh = \Qk$.  By Lemma \ref{collinearpointsinaplane} they are collinear $\iff \Tg_{12} + \Th_{23} = \Tk_{13}$.  It follows that $\Qg \cdot \Qh = \Qk \iff \Tg_{12} + \Th_{23} = \Tk_{13}$.

Set $\Qg=1$; then $1 \cdot \Qh = \Qk$ in $\fQ$ implies that $\Qh=\Qk$ and also that in $\fT$, $\Th_{23} = 0 + \Th_{23} = \Tk_{13} = \Th_{13}$, since $\Qh=\Qk$.  
This proves that two of the mappings agree; similarly, they agree with the third.  So, we have one function $\fQ \to \fT$ that carries $1$ to $0$ and preserves the operation.  It follows that the function is a monomorphism from $\fQ$ to $\fT^+$.  That establishes (d).
\end{proof}

\begin{thm}\label{bmain} 
Let $\PP$ be a projective plane and let $\fQ$ be a quasigroup.  
The following properties are equivalent:
\begin{enumerate}[{\rm(a)}]
\item $L_0(\fQ K_3)$ embeds as points in $\PP$.
\item $L(\fQ K_3)$ embeds as a dual affine $3$-net in $\PP$.
\item $L_0(\fQ K_3)$ embeds as lines in $\PP^*$.
\item $L(\fQ K_3)$ embeds as an affine $3$-net in $\PP^*$.
\item There exists a ternary ring $\fT$ such that $\PP = \PP_\fT$ and $\fQ$ is isostrophic to a subquasigroup of $\fT^+$.
\item There exists a ternary ring $\fT$ such that $\PP = \PP_\fT$ and $\fQ$ is isotopic to a subloop of $\fT^+$.
\item (If $\fQ$ is a loop.)  There is a ternary ring $\fT$ coordinatizing $\PP$ such that $\fQ$ is isomorphic to a subloop of $\fT^+$.
\end{enumerate}
\end{thm}

\begin{proof}
The equivalence of (a) and (b) is similar to that of (a) and (b) in Theorem \ref{parte1}(I), except that it is the common point of the main lines that serves as the extra point $\Tez$.

We obtain the other equivalences by duality and Proposition \ref{L:affnet}
\end{proof}

We want to compare the dual $3$-net embedding of this section to the synthetic orthographic matroid embedding of \orthopar.  We discuss the affine case; the projective case is similar.  
First, suppose in \orthoaff that $\Delta = K_3$ and $\bbA$ is an affine plane with a distinguished line $\bbA'$ in which $G(K_3)$ is represented by an embedding $\bz'$.  (In order for $\bz'$ to exist, the plane must have order at least 3.  That is not required in the projective case.)  Extend $\bbA$ to its projective plane $\PP$.  Choose $o=\bz'(e_{12})$, $u=\bz'(e_{23})$, and $i=\bz'(e_{13})$; also let $v$ be any point off $\bbA'$ and let $e$ be a point on $iv$ distinct from $i, v$.  We now have a coordinatization of $\PP$ by the ternary ring $\fT:=\fT(u,v,o,e)$ in which $uv$ is the ideal line, but all three lines (with $v$ omitted) are contained in the original affine plane $\bbA$.  In the notation of \orthopar, $\hE(K_3)$ is the union of those lines without $v$.  The three lines (without $v$) form a dual affine $3$-net $\cN^+$ and thence a biased expansion of $K_3$, in the notation of \orthopar $\Omega_0(\hE(K_3)) \cong \Omega(\cN^+) \cong \bgr{\fT^+K_3}$, and either \cite[Theorem 3.1\comment{TD:orthorep}]{BG6} or Lemma \ref{embeddingofTKintoP} tells us that 
$$
L_0\big(\Omega_0(\hE(K_3))\big) \cong M\big(\hE(K_3) \cup \{\he_0\}\big)
$$
by the natural correspondence $e\mapsto \he$ and $\Qez \mapsto \he_0$.

This matroid isomorphism implies that the balanced triangles of $\Omega_0(\hE(K_3))$ correspond to the collinear triples of points in $\hE(K_3)$ that are not contained in a main line; they also correspond to the concurrent triples of abstract lines in the abstract $3$-net $\cN(\fT^+)$, of which $\cN^+$ is a dual embedding in $\bbA$.  
Such a collinear triple is the simplest case of the following consequence of \cite[Corollary 4.5]{BG6}.  For $\hE \subseteq \hE(K_3)$, define $\Omega_0(\hE)$ as the spanning subgraph of $\Omega_0(\hE(K_3))$ whose edge set is $\hE$.

\begin{cor} \label{CD:orthoplane}
In an affine plane $\bbA$ let $G(K_3)$ be embedded in a line $\bbA'$.  If $\hE$ is a cross-closed subset of $\hE(K_3)$ that is not contained in a main line, then $\bgr{\fQ K_3} \cong \Omega_0(\hE)$, by the natural correspondence $e \mapsto \he$, for some subloop $\fQ$ of a principal loop isotope of $\fT^+$, and $M(\hE \cup \{\he_0\}) \cong L_0(\fQ K_3)$.  
\end{cor}

\begin{proof}
In light of the previous discussion, let $\Delta = K_3$ in \cite[Corollary 4.5]{BG6}.
\end{proof}

We draw the reader's attention to the fact that Corollary \ref{CD:orthoplane} is an affine proposition about a natural affine embedding, in contrast to Corollary \ref{CD:triplane}, the analog for the frame matroid and triangular $3$-nets.  We believe a frame matroid and a triangular $3$-net are essentially projective and their natural representations are by projective hyperplanes (lines, for the $3$-net); while a lift matroid and an affine $3$-net are essentially affine and their natural representations are by affine points.  That is an opinion but we think it is justifiable by the contrast between \mencev and \ortho and between Sections \ref{q frame} and \ref{q lift} as well as other reasons that are outside our scope here.

Affine $3$-nets are the planar instance of the synthetic affinographic hyperplane arrangements described in \orthoaffino.  This follows from the fact that the centers of an affine $3$-net are a representation of $G(K_3)$.  Throwing the centers to the ideal line $l_\infty$ in $\bbP$ gives a synthetic affinographic arrangement of lines in $\bbA := \PP \setminus l_\infty$.  The planar version of \cite[Corollary 3.12\comment{C:affinoexpansion}]{BG6} is the following property of affine $3$-nets.  

\begin{prop}\label{P:affsubnet}
Suppose $\cN$ is an affine $3$-net in a projective plane $\PP$ with centers in $l_\infty$; let $\bbA := \PP \setminus l_\infty$.  Then there are a ternary ring $\fT$ coordinatizing $\bbA$ and a subloop $\bar\fQ$ of $\fT^{*+}$ such that $\cN \cong \cN(\bar\fQ K_3)$.\footnote{We have not tried to decide whether the Hall--Dembowski definition $a+b := t(a,1,b)$ would result in replacing $\fT^*$ by $\fT$.}
\end{prop}

\begin{proof}
Choose $u,v$ to be two centers of $\cN$ and $o,e$ to be intersection points of $\cN$ so that $oe \wedge uv$ is the third center.  Specifically, let $u$ be the center for $\cL_{12}$, let $v$ be the center for $\cL_{23}$, and place the center for $\cL_{13}$ at $oe \wedge uv$.  This gives $\fT := \fT(u,v,o,e)$.  Let $\Omega$ be the biased expansion of $K_3$ corresponding to $\cN$.  We know that $\Omega = \bgr{\fQ K_3}$ for a quasigroup $\fQ$.  There is a balanced triangle in $\Omega$ that corresponds to $o$; by an isotopism transform $\fQ$ to a loop in which the edges of this triangle are labelled $1$.  
Recall that $\fT$ is in bijection with the ordinary part $ou \setminus u$ of $ou$ (the part not in the ideal line).  The points of $ou$ that belong to lines of $\cN$ are elements of $\fT$ and through the correspondence of lines in $\cN$ with $E(\Omega)$ they are labelled by elements of $\fQ$ with $o$ labelled $1$; that is, $0 \in \fT \leftrightarrow 1 \in \fQ$.  

Since $e$ is a point of $\cN$, $i = ou \wedge ve \in \fQ$; thus $i=1 \in \fT$ also belongs to $\fQ$ (but the $\fQ$-identity is $o=0$).  The slopes of the three pencils of $\cN$ are $0$ (parallels to $ou$), $1$ (parallels to $oe$), and $\infty$ (parallels to $ov$).

We now prove $\fQ$ has the same operation as $\fT^{*+}$; as an equation, $\Qg\cdot\Qh = \Qg+^*\Qh = t^*(1,\Qg,\Qh)$.  For that to be true we must embed the pencils with centers $v$ for $\cL_{12}$, $u$ for $\cL_{23}$, and $oe \wedge uv$ for $\cL_{12}$.  The property that $\Qg +^* \Qh = \Qk$, or $t^*(1,\Qg,\Qh) = \Qk$, in the primal plane is $t(\Qg,1,\Qk) = \Qh$, which means geometrically that $[\Qg,\Qh] \in \langle1,\Qk\rangle$.  Now, $[\Qg,\Qh]$ is the intersection of embedded lines $l_{12}(\Qg)$ and $l_{23}(\Qh)$, while $\langle1,\Qk\rangle$ is the embedded line $l_{13}(\Qk)$; so $\Qg \cdot \Qh = \Qk$ in $\fQ$.  This proves $\Qg \cdot \Qh = \Qg+^*\Qh$; thus, $\fQ$ is isomorphic to a subloop $\bar\fQ$ of $\fT^{*+}$.
\end{proof}

Proposition \ref{P:affsubnet} can be regarded as a variation (dualized) on the proof of the implication (a) $\implies$ (d) in Proposition \ref{L:affnet}.  It is a variation since here the centers are all placed at ideal points, while there the centers are points on the non-ideal line $ou$.  In other words, although we make different choices of locating the $3$-net in $\PP$ we have a similar conclusion about the associated quasigroup and ternary ring.  This is an interesting fact because the ternary rings of the two proofs may not be the same.

\subsection{Desarguesian planes}\label{q desarglift}\

In \cite[Section IV.4.1]{BG} Zaslavsky developed a system for representing a gain graph in a vector space over a skew field $\fF$ when the gain group is a subgroup of the additive group $\fF^+$.  \ortho generalizes that system to biased graphs by removing the use of coordinates.  We have just demonstrated that affine $3$-net embedding as either points or lines agrees with the system of \ortho.  It follows that the projective embedding of a quasigroup expansion $\fQ K_3$ by ordinary or dual affine nets (Section \ref{embed lift}) agrees with the coordinatized system of \cite[Section IV.4.1]{BG} when the quasigroup is a subgroup of $\fF^+$, or isotopic to such a subgroup, and the plane is the affine (or projective) plane over $\fF$.

Similarly, since the affinographic hyperplane representation of \orthoaffino agrees with the coordinatized affinographic representation of \cite[Corollary IV.4.5]{BG}, the affinographic line embedding of $L(\fQ K_3)$ referred to in Proposition \ref{P:affsubnet} agrees with that of \cite[Corollary IV.4.5]{BG} when the plane is Desarguesian over $\fF$ and $\fQ$ is a subgroup of $\fF^+$.


\section{Planarity of expansions and $3$-nets}\label{q planar}

A $3$-net in a projective plane must be either affine or triangular because its centers are collinear or not.  Consequently we may combine the results of the previous two sections into a complete characterization of the embeddability of a $3$-net into a projective plane, or equivalently the planar embeddability of the matroid of a quasigroup expansion $\fQ K_3$ of the triangle $K_3$.
We do so by dispensing with the something extra previously added that stabilized the representation: a prescribed basis in the full frame matroid and a special point in the extended lift matroid.  (We say simply ``matroid'' because the frame and lift matroids are identical: $G(\fQ K_3) = L(\fQ K_3)$; cf.\ Section \ref{matroids}.)  We treat this as two consequences: first, for a given quasigroup expansion; second, given a $3$-net that is already embedded in a projective plane.

\subsection{$3$-Net interpretation of the canonical representations}\label{netrep}\

We summarize the two kinds of $3$-net embedding as statements about a quasigroup expansion of $K_3$.  The matroid of a set of lines in $\PP$ can be defined as its matroid when considered as points in $\PP^*$.

\begin{prop}  \label{PD:netrep}
Suppose a quasigroup expansion graph $\Omega$ corresponds to a $3$-net embedded regularly in a projective plane $\PP$ as either lines (a net) or points (a dual net).  Then the matroid of the embedded net is a canonical frame representation of $\Omega$ if the net is triangular, a canonical lift representation if the net is affine.
\end{prop}

\begin{proof}  
This is essentially Proposition \ref{L:canonical} and an interpretation of \cite[Theorems 2.1\comment{TD:men} and 3.1\comment{TD:orthorep}]{BG6}.
\end{proof}

The non-main lines of $\cN^*$ are the cross-lines of the representation.  
Proposition \ref{PD:subnet} tells us that the $3$-subnets of $\cN$ correspond to the cross-closed subsets of the representation.

\subsection{Planarity of the quasigroup expansion matroid}\label{x planar}\

Now we collect the matroidal consequences, which are the culmination of this paper.

\begin{thm}[Planar Representation of a Quasigroup Expansion Matroid]\label{T:planar}
Let $\fQ K_3$ be a quasigroup expansion of $K_3$ and let $\PP$ be a projective plane.  The matroid $G(\fQ K_3)$ embeds as points in $\PP$ if and only if $\fQ$ is isotopic to a subquasigroup of $\fT^+$ or $\fT^\diamond$.  It embeds as points in $\PP^*$ if and only if $\fQ$ is isotopic to a subquasigroup of $(\fT^*)^+$ or $\fT^\times$.  In the former of each pair the representation is a canonical lift representation and in the latter it is a canonical frame representation.
\end{thm}

\begin{proof}
If the expansion is trivial (i.e., $\#\fQ=1$), the matroid is a three-point line, which embeds in every projective plane, and the quasigroup is isomorphic to the subquasigroups $\{1\} \subseteq \fT^\times$ and $\{0\} \subset \fT^+$.  Every representation is canonical.

If the expansion is nontrivial, the three edge lines of the representation are either concurrent or not, giving a canonical lift or frame representation, respectively.  
The representations of the full frame matroid are characterized in Theorem \ref{parte1} and those of the extended lift matroid are characterized in Theorem \ref{bmain}.
\end{proof}


\section{Classical and nonclassical questions on the existence of representations}
\label{questions}

If there exists a finite biased expansion of $K_3$ that has neither a frame nor a lift representation in a finite projective plane, its matroid would be a finite rank-3 matroid that cannot be embedded in a finite projective plane.  No such matroid is presently known to exist; this is a long-standing question \cite[Problem 14.8.1]{Oxley}.
A negative answer to Problem \ref{PrD:planerepb} or \ref{PrD:planerepl} would resolve the question by providing a nonembeddable example.  

Recall that a loop corresponds to a $3$-net, which corresponds to a Latin square $L$ (actually, an isostrophe class of Latin squares).  
It is well known that a $3$-net corresponding to $L$ is a complete triangular $3$-net in some (finite) projective plane, and a corresponding loop is the multiplicative loop of a ternary ring, if and only if $L$ belongs to a complete set of mutually orthogonal Latin squares.  
Hence, not every loop is the multiplicative loop of a ternary ring.  

Our interest in geometrical realizations of matroids, however, leads to a more relaxed question.  Our constructions show that a finite biased expansion of $K_3$ corresponds to a $3$-net.  
A classical open question is whether every finite loop is isotopic to a subloop of the multiplicative loop of some finite ternary ring.  Restated in terms of biased graphs this becomes a significant matroid question:

\begin{prob}	\label{PrD:planerepb}
Does every finite biased expansion of $K_3$ have a frame representation in some finite projective plane?
\end{prob}

We know by a theorem of Hughes \cite{HughesAML} that every countable loop embeds in (indeed, is) the multiplicative loop of some countably infinite projective plane, and hence every frame matroid of a countable biased expansion of $K_3$ is projective planar, embedding in a countable plane.    
The answer for finite loops and finite planes could well be different, but a negative answer would give a finite matroid of rank 3 that embeds in no finite projective plane.  (A positive result would not answer the matroid question since there might be other finite rank-3 matroids that are not finitely projective planar.)  
Section \ref{q frame} suggests that the problem of deciding whether $G(\gamma\cdot K_3)$ has a finite projective representation belongs more to the theory of $3$-nets than to matroid theory---which does not make it less interesting for matroids---because it shows that for a finite biased expansion $\gamma\cdot K_3$ and the corresponding 3-net $\cN$ the following statements are equivalent.  
\begin{enumerate}[{\rm(i)}]
\item  $G(\gamma\cdot K_3)$ has a frame representation in some finite projective plane.
\item  $\cN$ is embeddable as a triangular $3$-net (not necessarily complete) in some finite projective plane.
\end{enumerate}

We know by another theorem of Hughes \cite{HughesAML} that every countable loop is the additive loop of some countably infinite projective plane.  A second classical open question is whether every finite loop is isotopic to a subloop of the additive loop of some finite ternary ring.   
In the language of biased graphs: 

\begin{prob}	\label{PrD:planerepl}
Does every finite biased expansion of $K_3$ have a lift representation in a finite projective plane?
\end{prob}

This question reaches all the way to Latin squares, since by Section \ref{q lift} the following statements are equivalent.  Let $L$ be a Latin square corresponding to $\cN$.  
\begin{enumerate}[{\rm(i)}]
\item  $\gamma\cdot K_3$ has a lift representation in some finite projective plane.
\item  $\cN$ is embeddable as an affine $3$-net (not necessarily complete) in some finite projective plane.
\item  $L$ is a subsquare of one of a complete set of mutually orthogonal Latin squares of some finite order.
\end{enumerate}

Finally, we would like to apply our planarity results to all biased graphs of order 3.  
Theorem \ref{T:qx}, which says that any given finite biased graph of order 3 is contained in some finite quasigroup expansion of $K_3$, is a step in that direction but it leaves open the essential question of what quasigroup that can be; without an answer, our embeddability criteria cannot be applied.

\begin{prob}	\label{Pr:bgplanerep}
Show how to apply planarity results to biased graphs of order 3 that are not biased expansions.
\end{prob}


\end{document}